\documentclass[12pt]{amsart}
\usepackage{amsmath}
\usepackage{amsfonts}
\usepackage{amssymb}
\usepackage{amsthm}
\usepackage{hyperref}
\usepackage{enumerate}
\usepackage{cleveref}
\usepackage{mathrsfs}
\usepackage{lscape}
\usepackage{geometry}
\usepackage{multicol}
\usepackage{extarrows}
 \numberwithin{equation}{section}
\numberwithin{table}{section}

\theoremstyle{plain}
\newtheorem{theorem}{Theorem}[section]

\newtheorem{lemma}[theorem]{Lemma}
\newtheorem{proposition}[theorem]{Proposition}

\newtheorem{definition}[theorem]{Definition}
\newtheorem*{example}{Example}
\theoremstyle{remark}
\newtheorem*{remark*}{Remark}
\newtheorem{remark}{Remark}

\newcommand{\C}{\mathbb{C}}

\newcommand{\N}{\mathbb{N}}
\newcommand{\Q}{\mathbb{Q}}
\newcommand{\R}{\mathbb{R}}
\newcommand{\Z}{\mathbb{Z}}

\renewcommand{\H}{\mathbb{H}}

\newcommand{\T}{\mathbb{T}}
\newcommand{\TT}{\mathbb{T}}

\newcommand{\D}{\mathbf{D}}

\def\O{\mathcal O}

\def\({\left(}
\def\){\right)}

\newcommand{\frakp}{\mathfrak p}

\newcommand{\sm}[4]{\left(\begin{smallmatrix} #1 & #2 \\ #3 & #4\end{smallmatrix}\right) }
\newcommand{\zxz}[4]{\begin{pmatrix} #1 & #2 \\ #3 & #4 \end{pmatrix}}
\newcommand{\abcd}{\zxz{a}{b}{c}{d}}
\newcommand{\sabcd}{\sm{a}{b}{c}{d}}
\newcommand{\leg}[2]{\left( \frac{#1}{#2} \right)}

\newcommand{\SL}{\text {\rm SL}}

\newcommand{\into}{\hookrightarrow}
\newcommand{\onto}{\twoheadrightarrow}
\newcommand{\bij}{%
  \hookrightarrow\mathrel{\mspace{-15mu}}\rightarrow
}

\newcommand{\sgn}{\operatorname{sgn}}

\newcommand{\End}{\operatorname{End}}

\newcommand{\GL}{\operatorname{GL}}

\newcommand{\Gal}{\operatorname{Gal}}

\newcommand{\I}{\operatorname{i}}

\newcommand{\hatT}{\widehat{T}}

\begin{document}

\title{On $p$-adic harmonic Maass functions}
\author{Michael J. Griffin}
\address{Department of Mathematics, Brigham Young University, Provo, UT 84602}
\email{mjgriffin@math.byu.edu}

\begin{abstract} Modular and mock modular forms possess many striking $p$-adic properties, as studied by Bringmann, Guerzhoy, Kane, Kent, Ono, and others. 
Candelori developed a geometric theory of harmonic Maass forms arising from the de Rham cohomology of modular curves. In the setting of over-convergent $p$-adic modular forms, Candelori and Castella showed this leads to $p$-adic analogs of harmonic Maass forms. 

 In this paper we take an analytic approach to construct $p$-adic analogs of harmonic Maass forms of weight $0$ 
 with square free level.  Although our approaches differ, where the two theories intersect the forms constructed are the same. However our analytic construction defines these functions on the full super singular locus as well as on the ordinary locus.

As with classical harmonic Maass forms, these $p$-adic analogs are connected to weight $2$ cusp forms and their modular derivatives are weight $2$ weakly holomorphic modular forms. Traces of their CM values also interpolate the coefficients of half integer weight modular and mock modular forms. We demonstrate this through the construction of $p$-adic analogs of two families of theta lifts for these forms.
 
\end{abstract}

\maketitle 

\section{Introduction and statement of results}

Serre \cite{Serre} introduced the notion of a $p$-adic modular form as the limit of a sequence of modular forms with $p$-adically convergent $q$-expansions. This theory has been expanded by Dwork \cite{Dwork}, Katz \cite{Katz}, Hida \cite{Hida} and many others, filling out a beautiful picture in terms of the analysis and geometry of the modular curve and the Hecke algebra. 

The $p$-adic properties of mock modular forms are less well-studied. However Bringmann, Guerzhoy, Kane, Kent, Ono, and others (see for instance \cite{MockAsPAdic,PAdicCouplingHalf,ZagiersAdele,PAdicCoupling}) have demonstrated a number of striking examples.
In his masters thesis, Candelori \cite{Candelori} defined a $p$-harmonic differential for over-convergent $p$-adic modular forms, and by means of the de Rham cohomology defined $p$-adic analogs of $p$-harmonic Maass forms of weight $0$. In \cite{CandeloriGeometric}, Candelori presented a geometric theory of harmonic Maass forms. Later Candelori and Castella  \cite{CandeloriCastella} considered certain $p$-adic modular forms studied by Bringmann--Guerzhoy--Kane \cite{MockAsPAdic} associated to certain mock modular forms. They showed that these are over-convergent $p$-adic modular forms arising from the de Rham cohomology of modular curves in the same way as do harmonic Maass forms. From this geometric perspective, Candelori and Castella also reproduced results of Guerzhoy-Kent--Ono \cite{PAdicCoupling} on the $p$-adic coupling of mock modular forms with their shadow.

Similar $p$-adic properties are exhibited by both integer and half-integer weight mock modular forms. In fact, some evidence suggests many of these properties carry through certain lifts from integer weight forms to half integer weight forms, including for instance 
the lifts studied by Zagier in \cite{ZagierTraces}, which give the Fourier coefficients of half-integer weight forms as twisted ``traces'' over CM points of fixed discriminants. Unfortunately, we cannot simply apply these traces to general $p$-adic modular forms or over-convergent $p$-adic modular forms, since 
for sufficiently large discriminants these traces will eventually involve points with super singular $j$-invariants.   

Here we construct $p$-adic analogs of weight $0$ harmonic Maass forms from an analytic perspective, from sequences of classical modular functions. These functions converge on the full modular curve, including the super-singular locus. When the forms constructed have \emph{shadows} which are ordinary for $p$ (see Section \ref{SecIP}), 
 then these functions align with those studied by Candelori and Castella. In other cases, the $p$-adic harmonic Maass forms we construct are not standard $p$-adic modular forms, as witnessed by the presence of unbounded denominators in the $q$-series expansions. 

 If $K$ is a field over $\Q$, we will use the notation $M_k(N;K)$ and $M^!_k(N;K)$ respectively to denote the space of weight $k$ holomorphic and weakly holomorphic modular forms on the modular curve $Y_0(N)(K)$ (see Section \ref{SecNotation}). We also consider classical weight $0$ harmonic Maass forms as functions on $Y_0(N)(\C)\simeq\Gamma_0(N)\backslash\H,$ and denote the space of such functions by $H_0(N;\C)$.
Here and throughout, if $K$ is a number field we will use $\sigma$ to denote an infinite place of $K$, and $\frakp$ to denote a finite place lying over a rational prime $p$. We also denote the completion of $K$ at any place $\nu$ by $K_\nu$. With this notation we have the following.

  \begin{theorem}\label{ThmPhmf1}
  Suppose $N$ is a square free positive integer, $K$ a number field, and $\frakp$ a prime of $K$ over a rational prime $p$, not dividing $N$. There exists a $K_\frakp$ vector space $H_0(N;K_\frakp)$ 
   of $\frakp$-adically continuous functions $F^\frakp:Y_0(N)(K_\frakp)\to K_\frakp$, satisfying the following properties:
\begin{enumerate}
\item We have that $M^!_0(N;K_\frakp)\subset H_0(N;K_\frakp).$ \label{ThmFn_Inclusion}
\item Hecke operators and Atkin--Lehner involutions have well defined actions on $H_0(N;K_\frakp)$ and are endomorphisms.
\item Each function in $H_0(N;K_\frakp)$ has a well-defined $q$-expansion in $K_\frakp((q))$.\label{ThmFn_qexp}
\item The usual modular differential operator $\mathcal D$ acts on $H_0(N;K_\frakp)$, and the following  sequence is exact:
$$0\to K_\frakp\into H_0(N;K_\frakp) \xrightarrow{~\mathcal D~}  M^!_2(N;K_\frakp)\to M_2(N;K_\frakp)\to 0.
$$\label{ThmFn_Exact}
\end{enumerate}
  \end{theorem}  
The final map $M^!_2(N;K_\frakp)\to M_2(N;K_\frakp)$ in part (\ref{ThmFn_Exact}) above is obtained by demonstrating an orthogonal decomposition
 $M^!_2(N;K_\frakp)=\operatorname{\mathcal D} H_0(N;K_\frakp)\oplus M_2(N;K_\frakp).$
  
  Each function in the space $H_0(N;K_\frakp)$ is defined using two sequences of modular functions which converge on overlapping regions which cover $Y_0(N)(K_\frakp).$  
 We construct these sequences using the Hecke algebra. 
  The action of the Hecke operators, Atkin--Lehner operators and the differential operator $\mathcal D$ are obtained by acting component-wise on the defining sequences. 

A correspondence exists between certain functions in $H_0(N;K_\frakp)$ and in $H_0(N;\C).$
If $\mu$ is any embedding of fields $\mu: K\into L$, then $\mu$ extends naturally and uniquely to an embedding
\[
\mu : M^!_0(N;K)\into M^!_0(N;L), 
\]
so that $F^\mu|_{Y_0(K)}=F.$  
Similarly, if $L_1$ and $L_2$ are two fields containing $K$, then there is a natural equivalence relation $\simeq_K$ between forms $F_i\in M^!_0(N;L_i)$ which satisfy
\[F_1|_{Y_0(N)(K)}=F_2|_{Y_0(N)(K)}\in M^!_0(N;K).\]

For instance, suppose $K$ is a number field.  If $F^\sigma\in M^!_k(N;\C)$ has a $q$-expansion in $\Q((q))$, then there is a unique corresponding form $F^{\frakp}\in M^!_0(N;K_{ \frakp})$ so that 
\[
F^\sigma\simeq_\Q F^\frakp.
\]
We will establish a similar correspondence between a certain dense subspace of $H_0(N;\C)$ and $H_0(N;K_\frakp)$ respectively. This aim is complicated by the generally-expected  transcendence  of  the $q$-series coefficients and values at algebraic points of these functions. 
 However, this transcendence can be controlled in some key settings. 

For instance, suppose 
$g\in S_2(N;K)$ is a level $N$ newform with $q$-expansion given by $\sum_{n\geq1}a_g(n)q^n$.  Then there is a subspace $H^g_0(N;K,\sigma)\subset H_0(N;K_\sigma)$ of harmonic Maass forms $F$ satisfying the following properties:
\begin{enumerate}
\item The differential operator $\xi_0$  defined in (\ref{DefDXi}) acts on $F$ with 
\[\xi_0 F\in  K\cdot  \frac{g^\sigma}{\|g^\sigma\|},\] where $\|g^\sigma\|$ is the usual Petersson norm of $g^\sigma$.
\item The \emph{holomorphic part} $F^+$ of $F$ (see section \ref{SecMaass_Diff}) at each cusp has a $q$-expansion in $K_{\sigma}((q))$.
\item The principal parts of $F^+$ at each cusp is in $ K[q^{-1}].$ 
\end{enumerate}
Of course we could have written $\C$ or $\R$ instead of $K_{ \sigma}$ in part (2), but this notation suggests the direction in which we will generalize these properties for $H_0(N;K_\frakp)$. 

Suppose $F\in H^g_0(N;K,\sigma)$, with the $q$-expansion of $F^+$ at $\infty$ given by 
\[
F^+=\sum_{n\gg -\infty} a_F(n) q^n.
\]
The $K_\sigma$ vector space containing the coefficients $a_F(n)$ has rank at most $2$: there is some $\alpha^\sigma\in K_{\sigma}$ so that 
 \begin{equation}\label{EqnCoeffTrans}
 a_F(n)-\alpha^\sigma \frac{a_g(n)}{n}\in K.
 \end{equation}
In particular, if $a_g(n)=0$, then $a_F(n)\in K.$ The algebraicity over $K$ of $a_F(n)$ in this case was shown by Bruinier--Ono--Rhoades\cite{BOR};  Candelori \cite{Candelori} showed it is in $K$. 

This method of controlling the transcendental part using the coefficients of the cusp form underlies the $p$-adic coupling of mock-modular forms with their shadow as demonstrated by Guerzhoy--Kent--Ono \cite[Theorem 1.2]{PAdicCoupling}.
%

Similarly, in general we expect the values of $F$ at algebraic points to be transcendental, but there are important exceptions. The CM values of the $j$ function, known as  singular moduli, are algebraic. Zagier~\cite{ZagierTraces} showed that the twisted traces of these singular moduli are coefficients of certain weight $1/2$ and weight $3/2$ modular forms. 
These results have been studied and generalized in several directions, including for other modular functions of higher level, and for harmonic Maass forms. 
 Bruinier--Funke~\cite{BruinierFunke,BruinierFunkeTraces} and Alfes~\cite{AlfesKudlaMillson} have realized these trace maps as theta lifts obtained by taking the inner product of modular functions against certain non-holomorphic theta kernels. 
In many cases it is not hard to see that certain $p$-adic properties of $q$-series are propagated through the lifts. 
 Some of these properties have been explored by Bringmann--Guerzhoy--Kane \cite{PAdicCouplingHalf}. 
 
   Suppose $D$ and $\Delta$ are fundamental discriminants, both squares modulo $4N,$ with $D\Delta<0$.  Let $h$ be chosen so that $h^2\equiv \Delta D\pmod{4N}.$
We define a twisted modular trace $\widetilde {\mathbf{t}}_N (\Delta,D,h) (F)$
of values of $F$ at CM points of discriminant $D\Delta$  in equation (\ref{eqnIntroTrace}). 
 If $F$ is weakly holomorphic with rational coefficients, then this trace 
 is literally a trace over Galois conjugates of CM values of $D^{-\frac12}F$. If $F$ is not weakly holomorphic,
  the algebraicity is connected to the twisted $L$-function 
\begin{equation}\label{TwistedL}
L(g^\sigma,\Delta,s)=\sum_{n\geq 1}\frac{{\small\leg{\Delta}{n}}a_g(n)}{n^s}.
\end{equation}
By $L'(g^\sigma,\Delta,s)$, we denote the derivative in the $s$ variable. Results of Bruinier--Ono and Alfes can be used to control the transcendence of the traces. We package these results in the following proposition.

\begin{proposition}\label{PropTransValsInt}For all $D,\Delta,$ and $h$ satisfying the conditions above and \newline $F\in H^g_0(N;K,\sigma),$ we have that  
\begin{equation}\label{EqnVanL1}
\widetilde {\mathbf {t}}_N(\Delta,D,h) (F) \in K \iff  L(\xi_0 F,\Delta,1)L'(\xi_0 F,D,1)=0.
\end{equation}
More generally, there exist constants $\mathbf b_g(\Delta)\in K$ (dependent on $g$, but independent of the choice of $F$), and some $\boldsymbol \alpha^\sigma_{D}\in K_\sigma$ (independent of $\Delta$) so that 
\begin{equation}\label{EqnTraceTrans}
\widetilde {\mathbf {t}}_N(\Delta,D,h) (F)  - \boldsymbol\alpha^\sigma_{D}  \mathbf b_{g}(\Delta) \in  K.
\end{equation}
%
\end{proposition}
%
%
In particular, when $F$ is weakly holomorphic, we can take $\boldsymbol\alpha^\sigma_{D}=0.$
The $\mathbf b_g(\Delta)$ in the equation are coefficients of a certain weight $3/2$ modular form which corresponds to $g$ under the Shimura correspondence. The traces $\widetilde {\mathbf {t}}_N(\Delta,D,h) (F)$ themselves can be given in terms of coefficients of a weight $3/2$ harmonic Maass forms, and so (\ref{EqnTraceTrans}) strongly parallels (\ref{EqnCoeffTrans}). Equations (\ref{EqnVanL1}) and (\ref{EqnTraceTrans}) are tied together by work of Waldspurger~\cite{Walds} which shows that 
\[
\mathbf b_{g}(\Delta)=0 \iff L(g^\sigma,\Delta,1)=0.\]
 
If $N$ is square free, and not divisible by the prime $\frakp$, we define a space $H^g_0(N;K,\frakp)$ of $\frakp$-adic harmonic Maass forms which satisfy these same properties discussed above for $q$-expansions and CM values, with $\sigma$ replaced with $\frakp$. Additionally, these forms satisfy a natural correspondence with the space $H^g_0(N;K,\sigma),$ given by matching forms with the same principal parts at cusps. 

\begin{theorem}\label{ThmCorrespondence}
Assume the notation above, with $\sigma$ any infinite place of $K$, and $\frakp$ any finite prime of $K$ not dividing $N$. There exists a subspace $H^g_0(N;K,\frakp)\subset H_0(N;K_\frakp)$ satisfying a one-to-one correspondence with $H^g_0(N;K,\sigma)$. This correspondence maps each function $F^\sigma\in H^g_0(N;K,\sigma)$ to a corresponding $F^\frakp \in H^g_0(N;K,\frakp)$, satisfying the following properties.
\begin{enumerate}
\item The principal parts of the $q$-expansions of ${F^\sigma}^+$ and $F^\frakp$ at all cusps are 
equal. \label{ThmCorr_Principal}

\item If $a_{F^\sigma}$ and $a_{F^\frakp}$ denote the $n$-th coefficients of ${F^\sigma}^+$ and $F^\frakp$ respectively, then there exist $\alpha^\sigma\in K_\sigma$ and $\alpha^\frakp\in K_\frakp$, so that 
\[
\left(a_F^\sigma(n)-\alpha^\sigma \frac{a_g(n)}{n}\right) =
\left(a_F^\frakp(n)-\alpha^\frakp \frac{a_g(n)}{n}\right)\in K.
\]
\label{ThmCorr_AlphaCoeff}


\item 
 If 
  $L(\xi_0F^\sigma,\Delta,1)L'(\xi_0F^\sigma,D,1)=0,$
   then 
\[
\widetilde {\mathbf {t}}_N(\Delta,D,h) (F^\sigma)=\widetilde {\mathbf {t}}_N(\Delta,D,h) (F^\frakp)\in K.
\]
 More generally, there exist $\boldsymbol\alpha^\sigma_{D}\in K_\sigma$ 
 and $\boldsymbol\alpha^\frakp_{D}\in K_\frakp$ so that 
\[
\left(\widetilde {\mathbf {t}}_N(\Delta,D,h) (F^\sigma)  - \boldsymbol\alpha^\sigma_{D}\mathbf b_{g}(\Delta)\right) =
\left(\widetilde {\mathbf {t}}_N(\Delta,D,h) (F^\frakp)  - \boldsymbol\alpha^\frakp_{D}\mathbf b_{g}(\Delta)\right)\in K,
\]
where $\mathbf b_{g}(\Delta)$ is as above. 
\label{ThmCorr_AlphaTrace}

\item The above correspondence is equivariant with respect to the Hecke algebra  and Atkin--Lehner involutions.
\label{ThmCorr_HeckeEqui}
\end{enumerate} 
\end{theorem}
\begin{remark}\label{RmkCorr}
By transitivity, the correspondence given in Theorem \ref{ThmCorrespondence} extends to any two places $\mu$ and $\nu$ of $K$ finite or infinite which do not divide $N$.
Building on the earlier notation, we write $F^\mu\simeq_{K} F^\nu$ for functions which correspond in this manner. 
\end{remark}
\begin{remark}
Part \ref{ThmCorr_AlphaTrace} of Theorem \ref{ThmCorrespondence} follows from Theorem \ref{ThmThetaConv} which shows that the traces  interpolate $\frakp$-adic properties of half-integer weight harmonic Maass forms. 
\end{remark}

In Section \ref{SecIP} we will outline various structural results about the space $H_0(N;K_{\frakp})$. We will show that this space can be generated by the action of the Hecke algebra and Atkin--Lehner involutions acting on a single element. We show that the image under the modular derivative $\D$ is a distinguished subspace of $M_2^!(N;K_{\frakp})$, orthogonal to $M_2(N;K_{\frakp})$ under a natural $\frakp$-adic analog of the Petersson inner product. This space is distinguished by the $\frakp$-adic \emph{slopes} of the forms (see (\ref{DefSlope})). When the slope 
 is not negative, these align with the over convergent $p$-adic modular forms of Candelori and Castella's  theory. 

For primes $\frakp$ dividing $N$, we can define a similar space 
 which nearly satisfies Theorem \ref{ThmPhmf1}, and subspaces $H_0^g(N;K_\frakp)$ which satisfy Theorem \ref{ThmCorrespondence} parts (1),(2), and (4). However, aside from weakly holomorphic modular functions, the functions in these spaces are not well defined on the supersingular locus. Away from the supersingular locus, the results are $p$-adic modular forms. 
 The construction also converges on the supersingular locus, but the results seem to be incomplete. The limits branch depending on the cusp at which the expansion is taken, and the construction might be termed at best \emph{mock modular}. 
 We will primarily focus on the case $\frakp$ does not divide $N$, with a few exceptions in Theorem \ref{ThmIntegrality} below.

The correspondence described in Theorem \ref{ThmCorrespondence} and Remark \ref{RmkCorr} 
between places of $K$ raises the natural question if there is an adelic theory connecting these forms.  This question requires bounds on the denominators that can arise. 

\begin{theorem}\label{ThmIntegrality}
Let $(F^\nu)_\nu$ be a family of functions $F^\nu\in H_0(N;K,\nu)$ which are equivalent under $\simeq_K$, and such that the principal part of each $F^\nu$ at each cusp is defined over the ring of integers of $K.$ Denote the $q$-expansion of $F^{\nu}$ at a cusp $\rho$ by $\sum_{n\in \Z} a_\rho^{\nu}(n)q^n.$ Then there are integers $\mathcal M_N$,  $\mathcal R_N$, and $\mathcal B_{N,p}$ explicitly defined in section \ref{SecIntegrality}
, all independent of the family $(F^\nu)_\nu$ so that the following are true.
\begin{enumerate}
\item Suppose $\nu=\frakp$ is a finite place of $K$. Given any cusp $\rho$ and a positive integer $n =s^2 t$ with $t$ square free, we have
\[
v_\frakp\left( st\mathcal M_N \sqrt{\mathcal R_N} \cdot a^\frakp_\rho(s^2 t)\right) \geq 0.
\]
In particular, for fixed $n\in \Z$ and cusp $\rho,$ the vector $(a_\rho^\nu(n))_\nu$ is an adele. \label{ThmInt_coeff}

\item  Suppose $z\in Y_0(N)(K)$, with $v_\frakp(j(z))\geq 0,$ not supersingular at $\frakp$ if $\frakp$ divides $N$. Then 
$ \mathcal M_N \mathcal R_N \mathcal B_{N,p}\cdot F^{\frakp}(z)$
is $\frakp$-integral. 
\label{ThmInt_val}
\end{enumerate}
\end{theorem}

The bound on the denominators for evaluations is not sharp when $F^\frakp$ is weakly holomorphic, and may not be sharp in general. Improvements in this bound could be used to improve bounds for denominators appearing in algebraic coefficients of weight $1/2$ harmonic Maass forms.  

\begin{example}
Let $N=43$, $K=\Q$, and $\frakp=(3)$, and let $g$ be the unique newform for $\Gamma_0(43)$ with rational coefficients,
\[g=q -2q^2 - 2q^3 + 2q^4 -4q^5 + 4q^6 + q^9 + O(q^{10})
\]
There is a unique weight $0$ harmonic Maass form for $\Gamma_0(43)$ with the $q$-expansion at $\infty$ given by 
$$F^+(\tau)=q^{-1}+1.707216\dots q+1.792783\dots q^2+3.195188\dots q^3+ 
\dots,$$
which is invariant under the Fricke involution. Then 
$\D F=\mathcal F_\Q+\alpha_\C g$ where 
$$
\mathcal F_\Q=-q^{-1}+ 2 q+3q^2+9q^3+16q^4 +27q^5 +42q^6+ O(q^7)
\in M_2^!(43,\Q),
$$
and $\alpha_\C=-0.292783419\dots.$ 
The form $F^\frakp$ has the $q$-expansion
\begin{equation*}
\begin{split}
F^\frakp&=q^{-1}+\cdots212012_3 q+\cdots122110_3 q^2+\cdots102201_3 q^3 +O(q^4),
\end{split}
\end{equation*}
and satisfies $\D F^\frakp=\mathcal F_\Q+\alpha_\frakp g,$
where 
$\alpha_\frakp=\dots212010.1_3$. 
 Here we have represented each $3$-adic number in base $3$ 
  format so that, for instance, $$\alpha_\frakp=\dots212010_3=0\cdot 3^0+1\cdot 3^1+0\cdot 3^2+2\cdot 3^3+1\cdot 3^4+2\cdot 3^5+\dots.$$
  
Now let $\Delta=r=1.$ For $D$ and $h$ chosen as in Proposition \ref{PropTransValsInt}, the traces $\widetilde {\mathbf {t}}_N(\Delta,D,h)(F)$ are the coefficients of a weight $3/2$ vector valued harmonic Maass form $\vartheta_{\Delta,r}(F)$, as seen in Theorem \ref{ThmTraceCoeffs}. However for simplicity in this example we will consider a projection $\widehat{\vartheta_{\Delta,r}}(F)$ of this vector valued form to a scalar-valued form obtained by summing the vector components and multiplying by $-\frac12$.  Then ${\widehat{\vartheta_{\Delta,r}}(F)\in M^{!}_{3/2}(4\cdot 43;\C)}$, and lies in the Kohnen plus space.
The newform 
\[\widehat g=q^3+q^7-q^8-q^{12}+2q^{19}-q^{20}-2q^{27}-3q^{28}+O(q^{30})\in S_{3/2}(4\cdot43;\Q)\]
 maps to $f$ under the Shimura correspondence. 
We have $\widehat{\vartheta_{\Delta,r}}(F)=\widehat{\mathcal G_\Q}+a_\C ~\widehat g$,
where 
\[
\mathcal G_\Q=q^{-1}-1+q^7+q^8+q^{12}+q^{19}+q^{20}+2q^{27}+q^{28}+O(q^{30})\in M^!_{3/2}(4\cdot43;\Q), 
\]
and $a_\C=0.0663160686\dots$. 

The  traces $\widetilde {\mathbf {t}}_N(\Delta,D,h)(F^\frakp)$ are the coefficients of a $\frakp$-adic $q$-series ${\vartheta_{\Delta,r}}(F^\frakp)$ as seen in Theorem \ref{ThmThetaConv}. If $\widehat{\vartheta_{\Delta,r}}(F^\frakp)$ is the image under the same projection as above
, we find $\widehat{\vartheta_{\Delta,r}}(F^\frakp)=\widehat{\mathcal G_\Q}+a_\frakp ~\widehat g$,
where $a_\frakp=\dots00002.1_3$. 
Notice in this case we have a denominator of $3$. The constants
 in Theorem \ref{ThmIntegrality} are $\mathcal M_N=8$, $\mathcal R_N=2$ and $\mathcal B_{N,\frakp}=3,$ however since $F^\frakp\in H_0^g(43;\Q,\frakp),$ only $\mathcal B_{N,\frakp}$ contributes to the denominators.
\end{example}

The remainder of this paper will be organized as follows. In Section \ref{SecNotation} we review basic results about modular functions used throughout this paper. In Section \ref{SecMaass} we review the theory of harmonic Maass forms. Section \ref{SecConstruction} contains the construction of the $p$-adic harmonic Maass forms and the proofs of Theorem \ref{ThmPhmf1} 
 and Theorem \ref{ThmCorrespondence} parts (\ref{ThmCorr_Principal}),(\ref{ThmCorr_AlphaCoeff}), and (\ref{ThmCorr_HeckeEqui}).
 Section \ref{SecIP} contains additional results about the structure of the spaces of $\frakp$-adic harmonic Mass forms that we will find useful later. In Section \ref{SecIntegrality} we prove the integrality results for the $q$-series and values given in Theorem \ref{ThmIntegrality}. In Section \ref{SecHalfWt} we will review the theory of half-integer weight vector-valued modular forms and Hecke operators. In Section \ref{SecLifts} we will review the lifts connecting integral weight and half-integral weight forms and prove Proposition \ref{PropTransValsInt}. In Section \ref{SecPLifts} we extend the lifts studied in the previous section to the $\frakp$-adic harmonic Maass forms. 
Part (\ref{ThmCorr_AlphaTrace}) of Theorem \ref{ThmCorrespondence} will follow as a corollary to Theorem \ref{ThmThetaConv}. 

\section*{Acknowledgements}
This research was was supported by the National Science Foundation grant {DMS-1502390} and by  the European Research Council under the European Union's Seventh Framework Programme (FP/2007-2013) / ERC Grant agreement n. 335220 - AQSER. It was conducted during  postdoctoral work at Princeton University and at the Universit\"at zu K\"oln. 
 The author thanks these institutions for their support, with special thanks to his postdoctoral advisors Professor Shou Wu Zhang 
and Professor Kathrin Bringmann. 

The author is grateful to Claudia Alfes and Luca Candelori for comments and discussion on essential aspects of this work, and to Jonas Kaszian, Michael Mertens, Grant Molnar, and Michael Woodbury for their comments on earlier versions of this paper. The author would also like to thank the anonymous referee whose detailed comments have greatly improved the exposition of this paper.

\section{Modular functions}\label{SecNotation}

Throughout this paper, we will treat modular forms interchangeably as functions on elliptic curves, 
sections of line bundles over the modular curve, as formal $q$-series, 
 and in the complex case as functions in the complex variable $\tau$ in the upper half plane. We will treat harmonic Maass forms similarly.  

Given a model of an elliptic curve $E/\C$, let $\omega_1$ and $\omega_2$ be periods which generate the associated lattice, ordered so that $\tau_E:=\omega_1/\omega_2\in \H$. If $F$ is modular of weight $k$ and level $1$, then we have that
$$F(E)=F(\omega_1,\omega_2)=\left(\frac{2\pi \I}{\omega_2}\right)^{k} F(\tau_E).$$
If $F$ has level $N>0,$ then different choices of generators of the lattice may give different evaluations. 
A choice of level $N$ structure is a choice among the $\Gamma_0(N)$-equivalence classes of periods $\omega_1, \omega_2$ which generate the lattice. 

As usual any $2\times2$ rational matrix with positive determinant $\gamma=\sabcd\in \GL^+_2(\Q)$ acts on modular forms over $\C$ by 
\begin{eqnarray}
\nonumber F|_k\gamma\left(\omega_1,\omega_2\right)&:=&\operatorname{det}(\gamma)^{k/2} F\left( a \omega_1+b \omega_2 , c\omega_1+d\omega_2\right),\\
\label{EqnMatrixAction}
F|_k\gamma(\tau)&:=&(c\tau+d)^{-k} \operatorname{det}(\gamma)^{k/2} F\left(\frac{a \tau+b}{c\tau+d}\right).
\end{eqnarray}

Regardless of the field of definition of the forms under consideration, the matrix group $\GL^+_2(\Q)/\left(\Q\cdot \operatorname{I}_2\right)$ acts as an algebra of linear operators on modular forms, where the image of $\Gamma_0(N)$ acts trivially on level $N$ modular forms. This algebra contains both the Hecke algebra and the group of Atkin--Lehner involutions. 

Equivalently, we may consider the evaluation of modular forms algebraically. Let $\mathcal E(N;K)$ be the set of short Weierstraas models of elliptic curves over $K$ with a specified level $N$ structure. If $E\in \mathcal E(N;K)$, the evaluations of the Eisenstein series $E_4(E)$ and $E_6(E)$ can be read from the Weierstrass model. This suffices to evaluate any level $1$ meromorphic modular form. 
If $F$ is a modular function of level $N>1,$ then it is related to the $j$-function by a polynomial $\Phi_F(X,Y)\in K_F(Y)[X]$ for some field $K_F$, 
 defined by 
\begin{equation}\label{EqnModPoly}
\Phi_F(X,j):=\prod _{\gamma\in \SL_2(\Z)\backslash \Gamma_0(N)}\left( X-F|_0\gamma\right).
\end{equation}
Here the matrices $\gamma\in \SL_2(\Z)\backslash \Gamma_0(N)$ act by permuting the level $N$ structure of the input. Because the action 
by any matrix in $\SL_2(\Z)$ simply permutes the cosets, the coefficient functions must all be level $1$, and hence rational functions in $j$. Moreover, $\Phi_F(X,Y)$ 
 must be a perfect power of an irreducible polynomial. The level $N$ structure of $E$ then specifies an evaluation of $F(E)$ among the roots of $\Phi_F(X,j(E)).$

The geometry of the modular curve gives a more uniform characterization of the level $N$ structure.  The modular curve $Y_0(N)$ is a smooth affine curve over $\Q$, which satisfies the set bijection $\mathcal E(N;K)\bij K^\times\cdot Y_0(N)(K)$ and the isomorphism of Riemann surfaces $ Y_0(N)(\C)\simeq \Gamma_0(N)\backslash \H$ (See Theorem 13.1 of \cite{Silverman1}). 

We may fix a model 
$$Y_0(N)\simeq V\left(\Psi\right), \  \text{with } \ \Psi=\bigcup_i\{\Psi_i(\varphi_0,\varphi_1,\dots,\varphi_n)\},$$
so that the projection $(\varphi_0,\varphi_1,\dots,\varphi_n)\to (\varphi_0),$ gives the standard projection to $Y_0(1),$
each $\Psi_i(\varphi_0,{\varphi}_1,\dots,\varphi_n)\in \Z[\varphi_0,{\varphi}_1,\dots, \varphi_n],$ and each ${\varphi}_i$ satisfies a monic polynomial $\Phi_{{\varphi}_i}\in \Q[\varphi_0][X].$  
Modular functions on the curve can be represented as rational functions in the coordinates, with weakly holomorphic forms represented by polynomials:
$$M_0^!(N;K)\simeq K[\varphi_0,{\varphi}_1,\dots \varphi_n]/(\Psi).$$
 Each coordinate  ${\varphi}_i$ gives the value of an associated modular function $\tilde\varphi_i\in M_0^!(N;K).$ 
  Up to a linear change of variable 
  we may take $\tilde \varphi_0=j.$

If $K$ is a number field with ring of integers $\O$, we define the space of \emph{integral} modular functions
$$
M^!_0(N;\O):=\{F\in  M^!_0(N;K) \ : \ \Phi_F(X,j)\in \O[j][X]\}.
$$
It will be useful to fix a complete integral model of $Y_0(N)$, so that 
 $$M_0^!(N;\Z)\simeq \Z[\varphi_0,\varphi_1,\dots,\varphi_n]/(\Psi).$$
More generally, if $R$ is any fractional ideal of $K$ then we define the submodule $M_0^!(N;R)\subset M_0^!(N;K)$ by $M_0^!(N;R):=R\cdot M_0^!(N;\Z).$ 

We may impose topologies on the modular curve $Y_0(N)(K)$ in several natural ways. For instance, given an absolute value $|\cdot |_*$ on $K$, we may imposes a topology on $Y_0(N)(K)$ in terms of the distance between $j$-invariants of the points.
 We will make extensive use of certain important subsets $Y_0(N)(K)$ defined in terms of such topologies. If $\frakp$ is a prime of $K$, then
the \emph{$\frakp$-integral locus} is the set
$$\{z\in Y_0(N)(K) \ : \ |j(z)|_\frakp\leq1\}.$$ 
The $\frakp$-integral locus splits into two distinguished subsets: the \emph{supersingular locus}, 
$$
{SS_\frakp:=\left \{z\in Y_0(N)(K) \ : \ j(z) \pmod{\frakp}\text{ is supersingular } \right\},}
$$
 and the complement, the \emph{$\frakp$-ordinary locus}. 

\subsection{$q$-series and the Tate curve}
The $q$-expansion of a weakly holomorphic modular form $F$ corresponds to the evaluation of $F$ on a model of the Tate curve (see~\cite{Katz}). The various level $N$ models of the Tate curve correspond to the action of a matrix in $\SL_2(\Z)$ on a level $N$ modular form, and is related to the $q$-expansions of $F$ at the various cusps. 

The inequivalent cusps of $\Gamma_0(N)$ with $N$ square-free can be indexed by the divisors of $N$, with cusp $\{\frac ab\}$ indexed by $\delta=\frac{N}{\gcd(b,N)}$. The cusp $\infty$ then has index  $1$, while the cusp $0$ has index $N$. The Atkin--Lehner involutions $W_\delta$ permute these cusps. Here, $W_\delta$ can be represented by any integer matrix $\sm{\delta a}b{Nc}{\delta d}$ with determinant $\delta$. 
Then 
 $W_\delta$ swaps the cusp of index $D$ with that of index $\frac{D\delta}{(D,\delta)^2}$. 

We denote the standard $q$ series of a modular form $F$ at the cusp $\infty$ by $F(q)$. This corresponds to the model of the Tate curve $\textsc{Tate}_\infty(q)$ which satisfies 
\[
j_N(q)=j_N\left(\textsc{Tate}_\infty(q)\right)=j(q^N),
\]
where $j_N:= j|_0W_N$ is the image of $j$ under the Fricke involution $W_N:=\sm 0{-1}N0.$
In the complex case, this model corresponds to the usual Fourier expansion at $\infty$. 
More generally, each $\Gamma_0(N)$-model of the Tate curve $\textsc{Tate}_\infty(q)$ corresponds to the action of some right-coset representative  $[\gamma]\in \Gamma_0(N)\backslash \SL_2(\Z)$. The resulting action on the $q$-expansion of a modular form can be found by factoring  
\begin{equation}\label{EqnMatFactor}
\gamma=\abcd=W_\delta\zxz1j0\delta
\end{equation}
for some Atkin--Lehner involution $W_\delta$ with $\delta=\frac{N}{(N,c)}$, and $j\equiv d c^{-1} \pmod{\delta}$. Since $N$ is square-free in our case, $c$ and $\delta$ are co-prime, the factorization is well defined. 
 The action of an upper triangular matrix $\sm ab0d$ on a $q$-series is simply $q\to \zeta_{d}^bq^{a/d}$ with $\zeta_{d}$ a fixed primitive $d$-th root of $1$.
 Therefore 
\begin{equation}\label{EqnTate_gamma}
F\left(\textsc{Tate}_\gamma(q)\right)=F|_0W_\delta(\textsc{Tate}_\infty(\zeta_\delta^j q^{1/\delta})).
\end{equation}

The $q$-expansion principle (see~\cite{Katz}) allows us to use the $q$-expansion of a modular form $F\in M^!_k(N;K)$ to evaluate it at a curve $E\in \mathcal E(N;K)$. If $\nu$ is some place of $K$ such that $|j(E)|_\nu>1,$ then given any model $\textsc{T}(q)$ of the Tate curve, there are parameters $q_E$ and $\omega_E$ in $K_\nu$ with $|q_E|_\nu<1$ so that  
 \[
 F(E)=\omega_E^{k}F(\textsc{T}(q_E)).
 \]

Evaluating at the Tate curve shows that the the integral forms $M_0(N;\O)$ are exactly those level $N$ modular functions whose coefficients at all cusps are in $\O$.

\subsection{The Hecke algebra}\label{SecXHecke}
For our construction we will need an extension of the Hecke algebra
\[
\T_k^*(N;K)\subset \End_K(M^{!}_{k}(N;K)),
\] 
 generated by the standard Hecke operators $T_n$ for $(n,N)=1$, the Atkin $U_m$-operators for $m$ divisible only by primes dividing $N$
 , and the Atkin--Lehner involutions $W_\delta$ for $\delta \mid N$. For $N>1$ this is a non-commutative algebra. While operators with coprime index commute, the $U_{\ell^n}$ and $W_\ell$ operators for primes $\ell\mid N$ have non-trivial commutativity relations which can be worked out in terms of the action of matrices.
The $U_\ell$ operator satisfies 
\[F|_k \ell^{1-k/2}U_\ell=F\vert_{k}\sum_{j=0}^{\ell-1}\zxz{1}{j}{0}{\ell}.\]
The operators $(\ell^{1-k/2}U_\ell W_\ell)$ and $(W_\ell\ell^{1-k/2}U_\ell)$ both satisfy  the polynomial relation
$$x^2-(\ell-1)x-\ell=0.$$
The action of $(\ell^{1-k/2}U_\ell W_\ell)$ on a $q$-expansion is that of 
\begin{equation}\label{UW}
\ell U_\ell V_{\ell}+\ell^{k/2}W_\ell V_{\ell}-1,
\end{equation}
where as usual $V_\ell$ sends $q\to q^\ell.$ 



We will find it useful to define the operators 
\[\hatT_n:=\begin{cases}T_{n'} \ W_{\delta}U_{D}W_\delta &\text{ if } k\geq 2\\
n^{1-k}T_{n'} \ W_{\delta}U_{D}W_\delta &\text{ if } k\leq 0\end{cases}\]
where $\delta=(n,N)$, $n'$ is the greatest divisor of $n$ with $(n',N)=1,$ and $Dn'=n.$ 
A short exercise then shows that the weight $k$ and weight $2-k$ operators satisfy the same multiplicative relation
\begin{equation}\label{EqnHeckeMult}
\hatT_n\hatT_m=\sum_{\substack{d\mid (m,n)\\(d,N)=1}}d^{|k-1|}\hatT_{\frac{mn}{d^2}}.
\end{equation}
 In particular we have an isomorphism 
 $$\varphi_{k}:\TT^*_k(N;K)\xrightarrow[]{\sim} \TT^*_{2-k}(N;K).$$

The normalizations for the non-positive weight operators 
also preserve integrality of $q$-expansions. If $(n,N)=1$, this follows easily from the formula 
 in terms of the $U$ and $V$ operators,
\begin{equation}\label{EqnTHatCoP}
F|_k \hatT_n=F|_k \sum_{d\mid n}D^{|k-1|}U_dV_{\frac nd}.
\end{equation}
where $D=d$ or $\frac nd$ depending on whether $k\leq 0$ or $\geq 2$ respectively.
If $\ell\mid N$, 
 then the action on $q$-expansions of $\hatT_\ell$ can be worked out using (\ref{UW}). We find
\begin{equation}\label{EqnTHat}
F|_k W_\ell\ell^{1-k/2}U_\ell W_\ell=F|_k\left(\ell W_\ell U_\ell V_\ell-W_\ell+ \ell^{k/2}V_\ell\right).
\end{equation}

The normalizations also allow simpler commutativity relations with the modular differential operators $\D_{k-1}$ and $\xi_{2-k}$ defined in the next section. 
 

%
%

\section{Harmonic Maass forms}\label{SecMaass}
In this section we define harmonic Maass forms and lay out certain key properties that will be used later. 
 We begin by recalling the definition of harmonic Maass forms of weight $k\in 2\Z.$ 
 Here we set $\tau=x+\operatorname{i}y$ with $x$ and $y$ real, $y>0$, and set  $q=\operatorname{e}^{2\pi \operatorname{i} \tau}$. The weight $k$ hyperbolic Laplacian is defined by
\begin{equation*}
\Delta_k := -y^2\left( \frac{\partial^2}{\partial x^2} +
\frac{\partial^2}{\partial y^2}\right) + \operatorname{i}ky\left(
\frac{\partial}{\partial x}+\operatorname{i} \frac{\partial}{\partial y}\right).
\end{equation*}
\begin{definition}\label{DefHMF}
Let $\Gamma=\Gamma_0(N)$ for some 
$N$,  and let $ k\in 2\Z$. Then a real analytic function $F(\tau):\H\to \C$  is a {\it harmonic Maass form}  of weight $k$ for $\Gamma$ if:

\begin{enumerate}
\item  The function $F(\tau)$ is invariant under the slash operator so that 
$$F|_k \gamma =F$$ for every matrix $\gamma\in \Gamma.$
\item The function $F$ is harmonic so that $\Delta_kF=0;$
\item The function $F$ has a meromorphic \emph{principal part} at each cusp. That is, if $F_\rho$ is the expansion of $F$ at $\rho$, then
 there is some polynomial $P_\rho(q^{-1})\in\C[q^{-1}]$ and constant $C_\rho>0$ so that $F_\rho-P_\rho(q^{-1}) = O(e^{-C_\rho y})$ as $y\to \infty.$ 
  \label{MeroPP}
\end{enumerate}
\end{definition}
We denote the $\C$ vector space of weight $k$ harmonic Maass forms for $\Gamma_0(N)$ by $H_k(N;\C)$. The differential equation given by $\Delta_kF=0$ implies that harmonic Maass forms have Fourier expansions which split into two components: one part which is a holomorphic $q$-series, and one part which is non-holomorphic. 

\begin{lemma}[{Proposition 3.2 of \cite{BruinierFunke}}]\label{HMFparts}
Let $F(\tau)$ be a harmonic Maass form of weight $2-k<1$ for $\Gamma_0(N)$ as defined above. Then 
we have that 
\begin{displaymath}
F(\tau)=F^+(\tau)+F^{-}(\tau)
\end{displaymath}
where $F^+$ is the holomorphic part of $F$ or \emph{mock modular form},
 given by 
$$
F^+(\tau):=\sum_{n\gg -\infty} c_F^+(n) q^{n},
$$
and $F^{-}$ is the non-holomorphic part given by
$$
F^{-}(\tau) := \sum_{n \geq 1} c_{F}^-(n) \Gamma(k-1,4\pi y n) q^{-n}.
$$
\end{lemma}

\subsection{Differential operators and the Petersson inner product}\label{SecMaass_Diff}
Differential operators yield some important relations between spaces of harmonic Maass forms and weakly holomorphic modular forms of dual weight. Let $k\geq2$ be an even integer, and define the operators 
\begin{align}\label{DefDXi}
\D^{k-1}&:=\left(\frac{1}{2\pi i} \frac{\partial}{\partial \tau}\right)^{k-1}
&\text{and}&&   \xi_{k}&:=2iy^{k}\overline {\frac{\partial}{\partial \overline \tau}}.
\end{align}
These maps yield the exact sequences
\begin{align*} 
0\to M_{2-k}(N;\C) \hookrightarrow&H_{2-k}(N;\C)\xlongrightarrow[]{\D^{k-1}}S^\perp_k(N;\C)\to 0, \\
0\to M^!_{2-k}(N;\C)\hookrightarrow &H_{2-k}(N;\C)\xlongrightarrow[]{\xi_{2-k}}S_k(N;\C)\to 0 
\end{align*}
(See Corollary 3.8 of \cite{BruinierFunke}).
 Here, the space 
 $S_k^{\perp}(N;\C)$ is a distinguished subspace of $M^{!}_k(N;\C)$ consisting of those forms with vanishing constant term at all cusps and which are orthogonal to the cusp forms $S_k(N;\C)$ with respect to the regularized Petersson inner product described below.
 
 The $\D^{k-1}$ operator preserves integrality of coefficients, and so extends to a map 
 \[
 \D^{k-1}:M^!_{2-k}(N;K)\xlongrightarrow[]{\D^{k-1}}M^!_k(N;K).
 \]
The operators $\hatT_n$ commute with these differential operators. If $k\geq 2$ and $F\in H_{2-k}^{!}(N)$, then 
\begin{align}
\D^{k-1}\left(F|_{2-k}\hatT_n\right)&=\left.\left(\D^{k-1}F\right)\right \vert_{k}\hatT_n,  \label{EqnDHecke}\\
\xi_{2-k}\left(F|_{2-k}\hatT_n\right)&=\left.\left(\xi_{2-k}F\right)\right\vert_{k}\hatT_n.\label{EqnXiHecke}
\end{align}
 The same relations hold for the Atkin-Lehner involutions $W_\delta$.

The Petersson inner product $\langle\cdot,\cdot\rangle_k:M_k^!(N;\C)\times M_k(N;\C)\to \C$ is defined by the regularized integral
\begin{equation}\label{PIP}
\langle f,g\rangle_k:= [\SL_2(\Z):\Gamma_0(N)]^{-1}
\int_{\Gamma_0(N) \backslash \H}^{\operatorname{Reg}} f(\tau)\overline{g(\tau)}y^{k}\frac{dxdy}{y^2}.
\end{equation}
Borcherds' regularization of the inner product (see \cite{Bo1}) allows the inner product to make sense even in cases where we have growth towards the cusps. The normalization by the group index ensures that the inner product is independent of the level.

Bruinier--Funke~\cite{BruinierFunke} define a pairing  $\{\cdot,\cdot\}:S_{k}(N;\C)\times H_{2-k}(N;\C)\to \C$ connected to the inner product, defined by
\begin{equation}\label{DefPairing}
 \{g,F\} := \langle g,\xi_{2-k}F\rangle_k.
 \end{equation}
This pairing, and therefore the resulting inner product, can be computed in terms of the coefficients of harmonic Maass forms. 

\begin{theorem}[Proposition 3.5 of \cite{BruinierFunke}]\label{PairingFormula}
Let $F\in H_{2-k}^!(N;\C)$ and $g\in S_k(N;\C),$ and for each $D|N$, let 
 $F|_kW_{D}(\tau)=\sum_{n}a^+_D(n)q^n +F_D^-$ and $g|_k W_D(\tau)=\sum_nb_D(n)q^n.$  Then 
 \[
\{g,F\}= [\SL_2(\Z):\Gamma_0(N)]^{-1}
\sum_{D\mid N}\sum_{n\in \Z} a^+_D(-n)\cdot b_D(n).
\]
\end{theorem}
 
The pairing is a sum of the constant terms of the non-holomorphic weight $2$ modular forms $(F\cdot g)|W_D$. 
The formula presented here differs slightly from Bruinier and Funke's original statement which is given in terms of vector valued forms. The formula for the pairing is more easily recognized as a sum over cosets,
 \[
\{g,F\}=
[\SL_2(\Z): \Gamma_0(N)]^{-1}
 \sum_{\gamma \in \Gamma_0(N)\backslash \SL_2(\Z)} \sum_{n\in \Z} a^+_\gamma(-n)\cdot b_\gamma(n).
\]
Where $a^+_\gamma(n)$ and $b_\gamma(n)$ are coefficients of $F|_{2-k}\gamma$ and $g|_k\gamma$ respectively. The two definitions are equivalent, as can be seen by factoring the coset representatives as in equation (\ref{EqnMatFactor}). 

 An easy corollary of this theorem is that if $F\in H^{!}_{2-k}(N;\C)$ with $\xi_{2-k}F\neq 0$, then $F$ has a singularity at some cusp since $\langle\xi_{2-k}F,\xi_{2-k}F\rangle_k\neq 0$. 
Bruinier and Funke also show that given a Hecke eigenform $g\in S_k(N;\C)$ with coefficients in a field $K$, then there exists a harmonic Maass form $G\in H_{2-k}^!(N;\C)$ with $\xi_{2-k}G=\frac{g}{\langle g,g\rangle_k}$ whose principal parts at all cusps are defined over $K$.

\begin{proposition}\label{PropCGen}
Let $k\geq 2.$ The space of harmonic Maass forms $H_{2-k}(N;\C)$ is generated by the extended Hecke algebra $\TT_{2-k}^*(N;\C)$ acting on a single element.
\end{proposition}
\begin{proof}
There exists a harmonic Maass form in $H_{2-k}(N;\C)$ with a simple pole with leading coefficient $1$ at infinity, and with no other singularities. Call this form $P_*$. This form can be constructed by means of Maass Poincar\'e series \cite{Niebur,BringmannOnoMaassPoincare} or abstractly using the surjection $\D:H_0(N;\C)\onto S^{\perp}_2(N;\C),$ noting that $ S^{\perp}_2(N;\C)$ contains a form with such a principal part. 

As shown in section \ref{SecXHecke}, the operator $\hatT_n$ acts on a $q$-series so that 
$$P_*|_{2-k}\hatT_n =q^{-n}+O(1)$$
and will introduce no other poles. In particular, suppose $F\in H_{2-k}(N;\C)$ and has a $q$-expansion at each cusp given given by 
$$
F|_{2-k}W_\delta(\tau)=\sum_{n<0}a_\delta(n)q^n +O(1).
$$
Then the form
$$ F':= P_*|_{2-k}\sum_{\substack{n<0\\ \delta\mid N}} a_\delta(n)\hatT_nW_\delta
$$
will have the same singularities. Then $F-F'$ is a harmonic Maass form which is bounded at all cusps, and whose non-holomorphic part is $0$. Thus, by the remark following Theorem \ref{PairingFormula}, $F-F'$ is a holomorphic modular form of weight $2-k$. 

All that remains is to show that we can obtain the constant functions if $2-k=0$. Pick $\ell$ a prime divisor of $N$. We have two trace operators $\operatorname{Tr}_\ell$ and $\operatorname{Tr}_\ell'$ defined by
\begin{eqnarray}\label{EqnTrace1}
\operatorname{Tr}_\ell&:=&\sum_{\gamma\in \Gamma_0(N)\backslash \Gamma_0(\frac N\ell)} \gamma \ =1+\ell^{1-k/2}W_\ell U_\ell\\
\label{EqnTrace2}
\operatorname{Tr}_\ell'&:=&W_\ell\operatorname{Tr}_\ell =W_\ell+\ell^{1-k/2} U_\ell.
\end{eqnarray}
     
If $F$ is modular on $\Gamma_0(N\ell)$, then the image under either of these trace operators is modular on $\Gamma_0(N)$. Considering the $q$-expansion, we find that $P_*|_0\operatorname{Tr}_\ell'W_D(q)$ has no singularities, 
 and must be a constant. If it is non-zero, we are finished. Otherwise, replace $P_*$ with $P_*+1$. A short exercise shows that $\operatorname{Tr}_\ell$ acts on constants by multiplication by $\ell+1.$ 
\end{proof}

\subsection{More General harmonic Maass forms}\label{SecGenHMFS}
Recall that condition (\ref{MeroPP}) of \ref{DefHMF} allows only meromorphic singularities at cusps. In their original definition, Bruinier and Funke use a broader condition: 
\begin{itemize}
\item[(3*)] The function $F$ exhibits at most linear exponential growth at each cusp, so that if $F_\rho$ is the expansion of $F$ at $\rho$, then there is some constant $C_\rho\in \R$ satisfying $F_\rho = O(e^{C_\rho y})$ as $y\to \infty.$ 
\label{Sings}
\end{itemize}
 This broader definition allows for harmonic Maass forms with non-holomorphic principal parts. We denote this larger space by $\mathbb H_k(N;\C)$, however unless otherwise specified, by \emph{harmonic Maass form} we mean a form in $H_k(N;\C)$. 

The operators $\D^{k-1}$ and $\xi_{2-k}$ both map $\mathcal H_{2-k}(N;\C)$ on to the full space $M^!_k(N;\C).$ In particular, $\mathbb H_{2-k}(N;\C)$ contains forms whose holomorphic parts are \emph{cuspidal mock modular forms}, in that their images under $\D^{k-1}$ are cusp forms. We will come back to this point briefly in the next section.

\section{$\frakp$-adic harmonic Maass forms}\label{SecConstruction}

As we have seen in the previous section, the non-holomorphic parts of complex harmonic Maass forms are intimately connected to the differential operators $\D^{k-1}$ and $\xi_{2-k}.$ In particular, the derivative $\D$ annihilates the non-holomorphic part of weight $0$ harmonic Maass functions. On the $p$-adic integers, the Teichm\"uller character $\omega_p$ given by the limit of $p$-th powers
\begin{equation}\label{DefOmegap}
\omega_p(x):=\lim_{n\to \infty }x^{p^{n!}}
\end{equation}
 has a similar property. It converges $p$-adically for $|x|_p\leq 1$ to a function which is locally constant (and hence has derivative $0$), but which is not globally constant. The function $\omega_p(j)$ is an example of a $p$-adic function, defined on at least part of the modular curve which is not holomorphic and which has vanishing derivative. It also has a $q$-expansion of sorts. Although not convergent in $\Q((q)),$ the sequence of $q$-series $j(q)^{p^n}$ converges coefficient-wise to a constant term ($\equiv 744 \pmod p$). This convergence of the $q$-expansion hints at a way to use the Tate curve to extend the function $\omega_p(j)$ towards the cusps. Such an extension, of course does not make sense without some kind of regularization. This idea will be in the background of our constructions of $\frakp$-adic harmonic Maass forms. We will, however, make use of the Hecke operators $\hatT_p$ rather than explicit  polynomials of modular functions. 

  \begin{lemma}\label{HeckeLimit}\label{VanishingHecke}
Suppose $\frakp$ is a prime of $K$ not dividing $N$, and $f\in M^!_{0}(N;\O_\frakp).$
$$f_{n}=f |_{0} \hatT_{p^n},$$
 then the following are true:
\begin{enumerate}
\item For each divisor $\delta \mid N$ we have the $q$-series congruences 
\[\D f_n|_0W_\delta(q)\equiv 0\pmod{p^n}.\]
In particular, the $q$-expansions of $f_n$ at each cusp converge coefficient-wise to constant terms. If $C_\delta$ is the constant term of $f|_0W_\delta(q),$ then the constant term of $f_n|_0W_\delta(q)$ is congruent to $\frac{C_\delta}{1-p}\pmod{p^n}.$

 \item Suppose the completion $K_\frakp$ of $K$ at $\frakp$ is a finite Galois extension of $\Q_p$ with ramification degree $e$ and residue field $\mathbb F_{p^d},$ and let $c_e$ be the constant given in (\ref{EqRamConst}). 
Fix a complete integral model for $Y_0(N)(K)$ as in section \ref{SecNotation}.  Then there are locally constant functions $\hat f_{n}$  defined on the $\frakp$-integral locus which satisfy
 \[
v_p( f_{n}(z)-\hat f_{n}(z))\geq {n+c_e},
 \]
 whose indices depend only on the congruence class $n\pmod d$, and whose values depend only on the residues of the coordinates of $z \pmod \frakp$. 
\label{HL_ValLimits}
\end{enumerate}
\end{lemma}
The lemma extends naturally to the algebraic closure $\overline \Q_p$ or to $\C_p$. To extend part (2) we need only note that modulo a fixed $p$-adic valuation, the curve $E$ and the form $f$ must both be defined over a common finite extension of $\Q_p$. 

We will use this lemma repeatedly throughout our constructions, but will postpone its proof until section \ref{SecProofs}. The construction uses two regularizations for the convergence of sequences of modular forms. 
The regularizations used basically allow us to say that a sequence of modular forms whose values converge on the $\frakp$-integral locus and whose $q$-expansions at cusps converge to a form with finite principal parts should extend to some kind of form defined everywhere, even if the orders of the poles at cusps increase without bound. The first regularization in weight $2$ is fairly straightforward. The second regularization for weight $0$ relies on the Lemma and is more delicate as we work with sequences of forms whose $q$-expansions converge only coefficient-wise.

The construction of the space $H_0(N;K,\frakp)$ will begin with the construction of corresponding forms $F^\frakp$ for forms $F^\sigma\in H^g_0(N;K,\sigma),$ 
as in Theorem \ref{ThmCorrespondence}. After the constructions, we will prove the properties listed in Theorems \ref{ThmPhmf1} and 
\ref{ThmCorrespondence}, however many of these follow immediately from the construction.
The construction is uniform for each $F^\sigma\in H^g_0(N;K,\sigma),$ but does depend importantly on the newform $g$. It consists of sequences of operators $A_n$ in the Hecke algebra so that 
\begin{enumerate}
\item The forms $F^\sigma|_0A_n$ are weakly holomorphic and have coefficients in $K$.
\item The sequence of derivatives $(\D F^\sigma|_0A_n)^\frakp_n$ converges under the first regularization discussed in the next subsection to a form in $M_2^!(N;K_\frakp)$ with the same principal part at all cusp as $\D F^\sigma.$ 
\item The sequence of forms $(F^\sigma|_0A_n)^\frakp$ converges under the second regularization.
\item If $F^\sigma$ is weakly holomorphic, then the sequence of forms $ (F^\sigma|_0A_n)^\frakp$ converges to $F^\frakp.$
\end{enumerate}
The action of Hecke operators, Atkin--Lehner involutions and the derivative $\D$ commute with the operators $A_n$ without affecting convergence. 

It turns out that not only do such sequences of operators exists, but assuming their convergence properties arise from Lemma \ref{HeckeLimit}, the limits are almost unique. If the $p$-th coefficient of $g$ is divisible by $\frakp,$ then differences between limits may fall into a $1$ dimensional space spanned by a form whose derivative under $\D$ is a multiple of the newform $g$. Among these functions there is a natural choice for $F^\frakp\in H^g_0(N;K_\frakp).$ The cuspidal form can be viewed as an analog of a more general harmonic Maass form of the type discussed in section \ref{SecGenHMFS}. This form can be obtained from $F^\frakp$ by means of another limit of Hecke operators. This realizes the $p$-adic coupling between mock modular forms and their shadow as studied by Guerzhoy--Kent--Ono as an operation on $\frakp$-adic harmonic Maass forms, albeit using the $T_p$-Hecke operators rather than the $U_p$-operators. 

The construction for the cuspidal forms in general have worse convergence properties than do the forms in the space $H_0(N;K_\frakp)$. In particular, the limit fails to converge when the $p$-th coefficient is not divisible by $\frakp$--at least not by means of the regularizations considered here. The $q$-series do converge coefficient-wise, but we do not have convergence as functions on $\mathcal E(N;K_\frakp)$.

\subsection{Regularized convergence for $\frakp$-adic limits of modular forms}
\subsubsection{First regularization}\label{SecReg1} The first regularization for sequences of modular forms that we need is given below. 
The regularization also holds for a larger space of forms which may have poles in the supersingular locus, which we will use later. In this case, however, we will only concern ourselves with sequences which converge to weakly holomorphic functions. We define $\mathcal M^{p}_{k}(N;K_\frakp)$ to be the space of meromorphic modular forms of level $N$ over the field $K_\frakp$ with poles allowed at cusps and in the supersingular locus. 
  
  In the following lemma, we say a sequence of $q$-series, $f_m=\sum_{n\in \Z} a_m(n) q^n$ converges uniformly $\frakp$-adically if the $p$-adic limit
  \[
\lim_{m_1,m_2\to \infty}\left(\inf_{n\in \Z} v_\frakp\left(a_{m_1}(n)-a_{m_2}(n)\right)\right)\to \infty.
  \]
\begin{lemma}[First regularized convergence]
\label{RegCon1}
Suppose $(F_n)_{n\in \N}$ is a sequence of modular forms of $M^{!}_{k}(N;K_\frakp)$ with $k\geq 2$ whose  $q$-series at each cusp converge uniformly $\frakp$-adically 
 to a $q$-series with a bounded orders of poles. Then the limit 
\[F_\infty(q):=\lim _{n\to \infty}F_n(q)
\]
is the $q$-expansion for some  $F_\infty\in M^{!}_{k}(N;K_\frakp).$ Restricted to the $\frakp$-integral locus, we have uniform convergence as functions 
\[\lim_{n\to \infty}F_n(E) \to F_\infty(E).\]
 Moreover, there exists a sequence $(G_n)_{n\in \N}\subset M^{!}_{k}(N;K_\frakp)$ of modular functions which satisfies the following properties:
\begin{enumerate}
\item The maximum order of the poles at all cusps of the sequence $(F_n-G_n)_n$ is bounded as $n$ increases. 
\item The sequence of $q$-series $(G_n(q))_{n\in \N}$ converges uniformly to $0,$ and so the sequence $(F_n(q)-G_n(q))_{n\in \N}$ converges uniformly to $F_\infty(q).$
\item The sequence $(F_n(E)-G_n(E))_{n\in \N}$ converges to $F_\infty(E)$ uniformly on compact regions of $E\in\mathcal E(N;K)$, not including any cusp. 
\end{enumerate}

Given such a sequence $(F_n)_{n\in \N}$, we say that the sequence converges to $F_\infty$. 
\end{lemma}
A similar result also holds for a sequence $(F_n)_{n\in \N}\subset \mathcal M^{p}_{k}(N;K_\frakp),$ with a few modifications. First, for our purposes we will take it as a hypothesis that $F_\infty(q)$ is the $q$-expansion of a form in  $M^!_k(N;K_\frakp).$ Secondly, the forms $(F_n-G_n)_{n\in \N}$ will generally only converge outside the supersingular locus. 

\begin{proof}
We begin in the case $(F_n)_{n\in \N}\subset M^{!}_{k}(N;K_\frakp).$
If $k\geq 2$, then $M^!_k(N;K_\frakp)$ has an integral basis including forms $f_{d,m}$ for $m>0$ with $q$-expansions at cusps of the shape
\[
\left(f_{\delta,m}|_kW_D\right)(q)=\begin{cases} q^{-m}+O(1) &\text{ if } D=\delta\\
O(1) &\text{ otherwise.}
\end{cases}
\] 
Let $B$ be the order of the pole of $F_\infty.$ Due to the existence of the integral basis, we may construct forms $G_n$ so that each term in the sequence $(F_n-G_n)$ has poles with orders no greater than $B$. Moreover, we can do so using the basis elements $f_{\delta,m}$ with $m$ strictly larger than $B$. The convergence of the functions $F_n(q)$ imply that the principal parts of the $G_m$ go to $0$ and so the full $q$-expansions of the $G_m(q)$ functions converge to zero. Thus the sequence $(F_n(q)-G_n(q))_{n\in \N}$ converges to the same limit $F_\infty(q)$. Since we have bounded the order of poles, the forms converge in $M^!_k(N;K_\frakp)$, uniformly on compact regions not containing the cusps.

For $(F_n)_{n\in \N}\subset \mathcal M^{p}_{k}(N;K_\frakp)$ we proceed similarly. Using the integral basis we may construct forms $G_n$ so that $(F_n-G_n)$ converges to zero outside the supersingular locus. The space $\mathcal M^{p}_k(N;K_\frakp)$ is not closed under this regularized convergence as the weakly holomorphic forms are. The closure of this space includes all \emph{weakly holomorphic $p$-adic modular forms}. We will only need the lemma for sequences which converge to something weakly holomorphic.

Note that 
we can act on $F_\infty$ by the Hecke operators and Atkin--Lehner involutions by acting on $(F_n)_{n\in \N}$  term-wise without affecting convergence.
\end{proof} 

\subsubsection{Second regularization.}\label{SecReg2}
The second regularization for sequences of modular forms we need 
 is given below. 

\begin{lemma}[Second regularized convergence]
\label{RegCon2}
Suppose $(F_n)_{n\in\N}$ is a sequence of modular functions of  $M^{!}_{0}(N;K)$ which converges on the $\frakp$-integral locus, has a convergent constant term in the $q$-expansion at each cusp, and whose derivatives $(\D F_n)_{n\in \N}$ converge under the first regularized convergence to some function in $M^!_2(N;K_\frakp).$ 
Moreover, suppose there exists a sequence $(G_n)_{n\in \N}$ of modular functions of $\mathcal M^{p}_{0}(N;K_\frakp)$ which regularize convergence towards the cusps, satisfying the following properties:
\begin{enumerate}
\item The sequence $(G_n)_{n\in \N}$ converges uniformly to $0$ on the $p$-ordinary locus. 
\label{2reg1}
\item The maximum order of the poles at all cusps of $(F_n-G_n)_{n\in \N}$ is bounded as $n$ goes to infinity. \label{2reg2}
\item The sequence of derivative forms $(\D G_n)_{n\in \N}$ converges to $0$ outside the supersingular locus under the first regularized convergence.\label{DGn}
\label{2reg3}
\item The constant terms of the $q$-expansions of $G_n$ at all cusps converge to $0$.\label{GConst}
\label{2reg4}
\end{enumerate}
Then the function  
\[ F_\infty(E):= \begin{cases} \displaystyle\lim_{n\to \infty} F_n(E) & \text{ if } E \text{ is within the integral locus,}\\
\displaystyle\lim_{n\to \infty} \left(F_n-G_n\right)(E) & \text{ if } E \text{ is outside the supersingular locus,}
\end{cases}\]
is well defined on all of $\mathcal E(N;K),$ and is independent of the possible choices of sequences $(G_n)_{n\in \N}$ satisfying the properties above. Moreover $F_\infty(E)$ has a well-defined $q$-expansion at each cusp. 
\end{lemma}
As in the first regularization, we can act on $F_\infty$ by the Hecke operators and Atkin--Lehner involutions by acting on $(F_n)_{n\in \N}$  term-wise without affecting convergence.

If the sequence $(F_n)_{n\in \N}$ satisfies the theorem then we say it converges to $F_\infty$.
Notice in this case the existence of the sequence $(G_n)_{n\in \N}$ is a hypothesis of the lemma rather than a result as in the first regularization.  It is not hard to show that a sequence $(G_n)_{n\in \N}$ must exist which satisfies conditions (\ref{2reg2}), (\ref{2reg3}), and (\ref{2reg4}). That such a sequence also satisfies condition (\ref{2reg1}) implies certain constraints on the initial sequence $(F_n)_{n\in\N}$. 


\begin{proof} 
The $q$-expansion of $F_\infty$ at any cusp can be found, up to the constant term, by finding an anti-derivative of the limit of 
\[
\displaystyle\lim_{n\to \infty}\D F_n(q)=\displaystyle\lim_{n\to \infty}\D(F_n-G_n)(q).
\]
Coefficient-wise, we have convergence without the derivatives,
\[
\displaystyle\lim_{n\to \infty} F_n(q)=\displaystyle\lim_{n\to \infty}(F_n-G_n)(q).
\]

Given the $q$-expansion, we may use the Tate curve 
 to extend $F_\infty$ towards the cusps. The coefficients of $(F_n-G_n)$ may not have denominators bounded uniformly for all $n$, however the condition that $\D F_n$ converges restricts the denominators sufficiently that the value using the Tate curve will converge for any $q$ with $|q|_\frakp<1$. 
  If $(G_n)_{n\in \N}$ and $(G'_n)_{n\in \N}$ are any two such sequences satisfying the conditions, then conditions (\ref{2reg2}) and (\ref{2reg3}) imply that their sequence of differences $(G_n-G'_n)_n$ has bounded order poles and the $q$-series at each cusp converges coefficient-wise to $0$. Hence the sequence converges to $0$ on every model of the Tate-curve. Similarly, condition (\ref{2reg1}) implies that the differences must converge to $0$ on the ordinary locus. 
\end{proof}
  
\subsection{Construction of $H_0(N;K_\frakp)$}\label{SubSecConstruction}
We begin with the assumptions of Theorem \ref{ThmCorrespondence}. That is, let $g$ be a newform of level $N$ in $S_2(N;K)$ with $q$-expansion given by $g(q)=\sum_{n\geq 1} a_g(n)q^{n}$, and suppose $\frakp$ be a prime of $K$ not dividing $N$. Moreover, let $\sigma$ an Archimedean place of $K$, and $F^\sigma\in H^g_0(N;K,\sigma)$. Without loss of generality, assume that the principal parts of $F^\sigma$ at all cusps are $\frakp$-integral.

Let $\beta =\beta_g$ and $\overline\beta=\overline\beta_g$ be the roots 
 of the polynomial
\begin{equation}\label{BetaPoly}
x^2-a_g(p) x+p.
\end{equation}
Here we distinguish the roots so that $v_p(\beta)\leq v_p(\overline {\beta})$. Since the coefficients of $g$ are eigenvalues for Hecke operators, the Hecke relations imply that the $p$-th power coefficients of $g$ satisfy
 $$a_g(p^n)=\frac{\beta^{n+1}-\overline{\beta}^{n+1}}{\beta-\overline{\beta}}.$$
Note that so long as $v_p(\beta)< v_p(\overline {\beta})$, we have that $\beta\in K_\frakp.$ Otherwise, if the valuations are equal, much of the work of this section must take place over the field extension $K_\frakp(\beta)$. In this case, we will need a few extra steps at the end of the construction to assure everything is defined over $K_\frakp$ rather than the extension.

Consider the functions 
\[
F^\sigma|_0\left(1-\frac{\hatT_{p^n}}{a_g(p^n)}\right).
\]
These functions are all weakly holomorphic since the corresponding weight $2$ operators annihilate $g$ (applying, for instance (\ref{EqnXiHecke})). Since the principal parts of these functions are in $K$, all the coefficients are, and so these correspond to a family of functions in $M^{!}_0(N;K_\frakp).$
By Lemma \ref{HeckeLimit}, the $q$-series of these functions converge  $\frakp$-adically coefficient-wise. If $F^\sigma$ is holomorphic, then the lemma shows that coefficients converge $\frakp$-adically to those of $F^\sigma$--except for possibly the constant terms. Similarly, their values on the $\frakp$-integral locus do not generally converge, but instead oscillate near multiple limit points. This family of functions serves as a prototype for our purposes, but we must modify it to get solid convergence rather than oscillation.

Let $H^\sigma:=F^\sigma|_0(a_g(p)-\hatT_p)$ which is weakly holomorphic, and whose coefficients are in $K$ and are $\frakp$-integral since this is true for the principal parts. This being the case, it corresponds to a function $H^\frakp\in M^{!}_0(N;K_\frakp).$
As in Lemma \ref{HeckeLimit}, let $e$ be the ramification degree of $K_\frakp,$ and $d$ be the degree of the residue field. 

We will construct the $\frakp$-adic function $F^\frakp$ using limits of Hecke operators. 
For $n\geq0$, let  $B_{n}$ be the operator in $\TT^*_0(N;K(\beta))$ defined by
\begin{equation}\label{DefBn}
B_{n}:=\sum_{j=0}^{n-1} \beta^{-j-1} \hatT_{p^j}+ \sum_{j=0}^{d-1}\frac{\beta^{-n-j-1} \hatT_{p^{n+j}}}{1-\beta^{-d}}, 
\end{equation}
and let $\overline B_{n}$ be defined similarly, replacing $\beta$ with $\overline \beta.$ 

\begin{proposition}\label{PropFpConv}
Assume the notation above. Then the following are true.
\begin{enumerate}
\item the functions $F_{n}:=H^\frakp |_0B_{n}$ converge under the second regularization. 
\item If $v_p(\overline\beta)<1,$ the functions $\overline F_{n}:=H^\frakp |_0\overline B_{n}$ also converge under the second regularization.
\item If $F^\sigma$ is weakly holomorphic, the limits described above both give the corresponding $\frakp$-adic function $F^\frakp \in M^!_0(N;K,\frakp).$
\end{enumerate}
\end{proposition}
As can be seen from the proof, if $v_p(\overline\beta)<1/2,$ then $F_n$ will converge more quickly than will $\overline F_n.$ Moreover, Hensel's lemma applied to the polynomial $x^2-x+\frac{p}{a_g(p)^2}$ (which has roots $\frac{\beta}{a_g(p)}$ and $\frac{\overline \beta}{a_g(p)}$) shows that  $\beta\in K_\frakp,$ and so each of the forms $F_n\in M^!_0(N;K,\frakp).$
When $v_p(\beta)=v_p(\overline\beta)=1/2,$ the distinction between $\beta$ and $\overline\beta$ is arbitrary. Moreover, Hansel's lemma may no longer apply and so $\beta$ may not lift to an element in $K_\frakp$ and $F_n$ may not live in $M^!_0(N;K,\frakp).$ However $\tfrac12(F_n+\overline F_n)$ will live in $(M^!_0(N;K,\frakp).$ 

In light of these observations, if $F^\sigma$ is not weakly holomorphic, it is natural to define $F^\frakp$ in terms of these limits as follows.

\begin{definition}\label{Def_Fp}
Assume the notation above. If $F^\sigma$ is not weakly holomorphic, define
\[
F^\frakp :=\begin{cases}\lim_{n\to \infty} F_n & \text{ if } v_p(\beta)<1/2,\\ 
\lim_{n\to \infty} \tfrac12(F_n +\overline F_n)& \text{ if } v_p(\beta)=1/2.\\
\end{cases}
\]
\end{definition}

\begin{proof}[Proof of Proposition~\ref{PropFpConv}]

The multiplication rule (\ref{EqnHeckeMult}) implies
\[
\hatT_{p^n}\hatT_p=\hatT_{p^{n+1}}+p\hatT_{p^{n-1}}.
\]
Using this and telescoping the resulting sums, we find 
\begin{equation}\label{prodH}
\begin{split}
A_n&:=(a_g(p)-\hatT_p)B_{n}\\
& \ =1+\frac{\beta^{-n-d}}{1-\beta^{-d}}
\left( -\overline\beta(\hatT_{p^{n-1}}-\hatT_{p^{n+d-1}})
+(\hatT_{p^{n}}-\hatT_{p^{n+d}})\right).
\end{split}
\end{equation}
The difference between sequential $B_n$ is 
\begin{equation}\label{DiffBn}
B_{n+1}-B_{n}= \beta^{-n-d-1}\frac{ \hatT_{p^{n+d}}-\hatT_{p^{n}}}{1-\beta^{-d}}.
\end{equation}

Applying Lemma \ref{HeckeLimit}, we see that the $q$-series $F_{n}(q)=H^\frakp |_0B_{n}(q)$ at each cusp converges coefficient-wise, except possibly the constant term. The action of $\hatT_p^{n}$ on a constant terms is simply multiplication by $\frac{1-p^{n+1}}{1-p}\equiv \frac{1}{1-p}\pmod{p^{n+1}}$, so the difference of constant terms must go to zero as well. If $C_\rho$ is the constant term of $F^\sigma$ at the cusp $\rho$, then the constant term of $F_{n}$ at $\rho$ can be found using (\ref{prodH});
\[
C_\rho\left(1+\frac{\beta^{-n-d}}{1-\beta^{-d}}\left( - \overline\beta\frac{-p^n+p^{n+d}}{1-p}
+ \frac{-p^{n+1}+p^{n+d+1}}{1-p}\right) 
\right) \equiv C_\rho\pmod{(p/\beta)^{n+1}}.
\]
  If $F^\sigma$ is weakly holomorphic, then (\ref{prodH}) and Lemma \ref{HeckeLimit} tells us that the coefficients of $F_{n}$ at any cusp eventually converge to those of $F^\sigma$.
    
Let $\widehat H_i$  be the locally constant functions from part (\ref{HL_ValLimits}) of Lemma \ref{HeckeLimit}, so that $\widehat H_j(E)\equiv H^\frakp|_0\hatT_p^j(E)\pmod {p^{j+c_d}}$ for any curve $E$ in the $\frakp$-integral locus. Then given such a curve $E$, we have 
\begin{equation}\label{ValDiff}
\begin{split}
\left(F_{n+1}-F_{n}\right)(E)
&\equiv \beta^{-n-d-1}\frac{ \widehat H_{n+d}-\widehat H_{n}}{1-\beta^{-d}}
\equiv 0  \ \ \pmod{p^{n+c_e-b_\beta}\beta^{-n-1}}
\end{split}
\end{equation}
where $b_\beta =v_p(1-\beta^d).$ 
Thus the functions converge. 
Notice since $\beta$ is not a root of unity, $b_\beta$ can be bounded independent of $d$. Similarly,  $d$ may be replaced with any positive multiple without altering the result modulo $p^{n+c_e-b_\beta}\beta^{-n-1}$. In particular, expanding the field of definition does not alter the limit.

We are nearly ready to show the sequence of functions $(F_{n})$ converge under the second regularization. However we still need suitable functions $G_{n}\in \mathcal M^{p}_0(N;K).$ The construction is not difficult, but the functions must satisfy several properties so there are several short steps involved.
 
  Choose some constants $\lambda$ and $P$ so that
\[
P:=p^{\lambda-1}(p-1)>2
\]
 and let $\mathcal A:=E_{P}\in M_P(1;K_\sigma)$ be the standard Eisenstein series of weight $P$, which satisfies $E_{P}(q)\equiv 1\pmod{p^\lambda}.$ The function $\mathcal A^{12}/\Delta^{P}$ forms a polynomial in $j$ whose roots $\pmod{p}$ are exactly the supersingular $j$-invariants. The space $M^{!}_{P}(N;K_\frakp)$ contains forms with any given order of pole, and so it contains a form
 $f_1$ with the same principal part as $F^{\sigma}\cdot \mathcal A.$ The space also contains Eisenstein series with a constant term at any single fixed cusp, and so we may modify the constant terms of $f_1$ to obtain a function $f_2$ so that $f_2/\mathcal A$ has vanishing constant terms at every cusp. Finally, if the coefficients of $f_2$ are not $\frakp$-integral, then modulo the $\frakp$-integers, they reduce to the coefficients of some cusp form in $S_{P}(N;K)$. By subtracting this cusp form, we obtain a function, $f_3$, which has $\frakp$-integral coefficients. Then let $F^*=f_3/\mathcal A \in \mathcal M^{p}_{0}(N;K)$. By construction, this form has $\frakp$-integral coefficients and has the same principal parts as those of $F^{\sigma}$ at all cusps, except for the constant terms which are all $0$.

We can now define 
\[G_{n}:=F^*|_{0}\left((a_g(p)-\hatT_{p})B_{n}-1\right).\]
Notice the principal part of $F_{n}-G_{n}$ is always the same as that of $F_g^{\sigma},$ apart from the constant term. The constant terms of all the $G_{n}$ are $0$, and we have already seen that the constant terms of $F_{n}$ at cusps converge to those of $F_g^{\sigma}.$ Using Lemma \ref{HeckeLimit} for $\mathcal M^{p}_0(N;K)$ and (\ref{prodH}), we see that the $G_{n}$ converge to $0$ on the $\frakp$-ordinary locus, and the derivatives $\D G_{n}(q)$ converge to $0$ under the first regularization. This last piece is required to show that the sequence $(F_{n})_n$ converges under the second regularization, with the sequence $(G_{n})_n$ regularizing the convergence near the cusps.  

This argument above goes through mutatis mutandis for $\overline F_n$.  
Let the operators 
 $\overline A_n$ and 
  functions  
  $\overline G_{n}$ be defined by swapping $\beta$ and $\overline\beta$ in the definitions of $A_n$ and $G_n$ above. As long as  $v_p(\beta)>0,$ the functions $\overline F_{n}$ will converge under the second regularization, though more slowly than will the $F_{n}$, as seen by swapping $\beta$ for $\overline\beta$ in the modulus of (\ref{ValDiff}). 
\end{proof}
 
   The sequence of differences $\frac{(F_{n}-\overline F_{n})}{\beta-\overline\beta}$ converges to a function which has no principal part. If our initial form is weakly holomorphic, the difference is identically $0$. Otherwise, its derivative is a weight $2$ cusp form. Using the Hecke relations for $g$, it is easy to see that the derivative must be a multiple of $g$, since these same relations send $F^\sigma$ to a weakly holomorphic form, and commute with the operators $B_n$. In particular, we may view the limit of the sequence $\left(\frac{(F_{n}-\overline F_{n})}{\beta-\overline\beta}\right)_n$ as an analog of a cuspidal mock modular form related to $g$, as described in Section \ref{SecGenHMFS}.

This cuspidal form, when it exists, can be recovered from the function $F^\frakp$ which has singularities. For instance 
\[
(F_{n}-\overline F_{n})=\lim_{n\to\infty}F^\frakp|_0(A_{n}-\overline A_{n})
\]
since $F^\frakp|_0(a_g(p)-\hatT_p)=H.$ This operation realizes a modified version of the $p$-adic coupling of between mock modular forms and their shadows as studied by Guerzhoy--Kent--Ono (using the $T_p$ Hecke operators rather than the $U_p$ operators) as an operation on $\frakp$-adic harmonic Maass forms. 

The construction of the cuspidal form fails when $v_p(\beta)=0$. In that case the functions $\overline F_{n}$ converge coefficient-wise as $q$-series at any cusp, but not by means of the regularizations considered here as functions on $\mathcal E(N;K_\frakp)$.

If $K$ contains the coefficients of every newform $g$ of level dividing $N$ and $\nu=\sigma$ is any infinite place of $K$, then the full space $H_0(N;K_\nu)$ can be decomposed as 
 \begin{equation}\label{eqn_HMF_decomp}
 H_0(N;K_\nu)=\bigoplus_{\substack{g\in \mathcal S_M\\ \delta~\mid ~N/M}}H^g_0(M;K_\nu)|W_\delta.
 \end{equation}
 Here, we use $\mathcal S_M$ to denote the set of primitive newforms of level $M$, where $M$ is a divisor of $N$.
We define  $H_0(N;K_\frakp)$ by (\ref{eqn_HMF_decomp}), taking $\nu=\frakp.$
Since the $g|_2 W_\delta$ are all linearly independent, the intersection of any of these spaces must consist only of weakly holomorphic modular forms.

\subsection{ A construction when $\frakp\mid N$}\label{SubSecConstPN}
If $g$ is as above, but $\frakp$ divides the level, the situation is somewhat less straightforward. In this case, the supersingular locus divides $\mathcal E(N;K^\frakp)$ into two regions, one near the cusps with denominators divisible by $p$, and one near the cusps with denominators coprime to $p$. The two regions can be distinguished by, for instance, whether or not
\[
 \lim_{n\to \infty} j|_0(p U_p)^n(E) = 0.
 \]

The cusp form $g$ must satisfy $g|_2W_p=\lambda_p g$ with $\lambda_p=-a(p)=\pm1.$
As before, we may construct functions 
\begin{align*}
H^{\sigma}_{1,n}&:=F^{\sigma}|_0\left(1-(pU_p)^n\right)\\
H^\sigma_{2,n}&:=F^{\sigma}|_0\left(1-W_p(pU_p)^nW_p\right)
\end{align*}
which are weakly holomorphic, and whose coefficients at each cusp are in $K$ and are $\frakp$-integral. The corresponding functions $H^{\frakp}_{1,n}$ and $H^{\frakp}_{2,n}$ in $M^!_{0}(N;K_\frakp)$ converge as $p$-adic modular forms defined respectively on each of the two regions of $\mathcal E(N;K^\frakp)$. It turns out that the two sequences of functions also converge on the supersingular locus, but they necessarily disagree at some point unless $F^\sigma_g$ is weakly holomorphic. 

For our purposes, we will define the function $F_g^\frakp$ for $E \in \mathcal E(N;K^\frakp)$ not supersingular by the piecewise limit  
\[
F^\frakp_g(E)=\begin{cases} \lim_{n\to \infty} H^{(1)}_n(E) &\text{ if }\displaystyle\lim_{n\to \infty} j|_0(p U_p)^n(E)=0 \\
\lim_{n\to \infty} H^{(2)}_n(E) & \text {otherwise.} \end{cases}
\]
The $q$-expansions for $F^\frakp$ are defined by 
\[
F^\frakp_g|W_D(q)=\begin{cases} \lim_{n\to \infty} H^{(1)}_n|W_D(q) &\text{ if }(p, D)=1 \\
\lim_{n\to \infty} H^{(2)}_n|W_D(q) & \text {otherwise.} \end{cases}
\]

\subsection{Some Proofs}\label{SecProofs}
Here we prove Lemma \ref{HeckeLimit}, Theorem \ref{ThmPhmf1},  and Theorem \ref{ThmCorrespondence} parts (\ref{ThmCorr_Principal}),(\ref{ThmCorr_AlphaCoeff}), and (\ref{ThmCorr_HeckeEqui})
\begin{proof}[Proof of Lemma \ref{HeckeLimit}]
Since $p$ does not divide $N,$ the action of the Hecke operator $\hatT_{p^n}$ on $q$-series is given by
\[
\sum_{i=0}^n(pU_p)^{n-i}V_{p^i}.
\]
This acts on constants by multiplication by $\frac{1-p^{n+1}}{1-p}.$ Moreover, we see
\begin{equation}\label{EqnTHatCong}
f|_0\hatT_{p^n}(q)\equiv f|_0\hatT_{p^{n-1}}(q^p)\pmod{p^n}.
\end{equation}
 An induction argument proves the claim about the derivative.

The second part again follows from the congruence (\ref{EqnTHatCong}). 
As in Section \ref{SecNotation}, fix an integral model of $Y_0(N)$ with coordinates corresponding to functions $\hat\varphi_i,\dots,\varphi_M\in M^!_0(N;\O_\frakp)$ so that $M^!_0(N;\O_\frakp)= \O_\frakp[j, \tilde\varphi_1,\dots,\tilde\varphi_M],$ and note the each of these functions takes on values in $\O_\frakp$ for $z\in \O$ . Then 
$$f=r_0(j,\tilde\varphi_1,\dots,\tilde\varphi_M)$$ for some polynomial $r_0$ with coefficients in $\O_\frakp$.
 The congruence (\ref{EqnTHatCong}) implies that 
\[
f|_0\hatT_{p}(q)\equiv r_0(j^p, \tilde\varphi_1^p,\dots,\tilde\varphi_M^p)\pmod{p},
\]
or more generally, there are polynomial $r_i$ with $\frakp$-integral coefficients so that
\begin{align}
f|_{0} \hatT_{p^{~}}=& r_0(j^p, \tilde\varphi_1^p,\dots,\tilde\varphi_M^p)+p\cdot r_1(j, \tilde\varphi_1,\dots,\tilde\varphi_M)\nonumber\\
f|_{0} \hatT_{p^{2}}=&r_0(j^{p^2}, \tilde\varphi_1^{p^2},\dots,\tilde\varphi_M^{p^2})+p\cdot r_1(j^p, \tilde\varphi_1^p,\dots,\tilde\varphi_M^p)+p^2 \cdot r_2(j, \tilde\varphi_1,\dots,\tilde\varphi_M)\nonumber\\
&\vdots\nonumber\\
f|_{0} \hatT_{p^{n}}=&\sum_{i=0}^n p^{i}\cdot r_i(j^{p^{n-i}}, \tilde\varphi_1^{p^{n-i}},\dots,\tilde\varphi_M^{p^{n-i}}).\label{EqnHeckeConv}
\end{align}
The locally constant limit functions $\hat f_n$ are given by
\[
\widehat f_n(E):=\sum_{i=0}^\infty p^{i}\cdot r_i^{p^{n-i}}
(\omega_p^{p^{n-i}}(j(z)),\omega_p^{p^{n-i}}(\tilde\varphi_1(z)),\dots,\omega_p^{p^{n-i}}(\tilde\varphi_M(z))).
\]
Here $\omega_p(x)$ is the character defined in (\ref{DefOmegap})
 which depends only on the residue ${x\pmod \frakp}$. 
 Note that the residue field has order $p^d$ so the $\hat f_n$ are periodic in $n$ with order $d.$
  
In order to see how closely $\hat f_n(z)$ approximates $f_n$, notice that any 
$t\in \O_\frakp$ can be written as $t=\zeta+b$ where $\zeta=\omega_p(t)$ satisfies the polynomial equation $\zeta^{p^d}-\zeta=0,$  and $v_p(b)>0.$ 
 We must find a lower bound for the valuation of 
\[
 t^{p^s}-\zeta^{p^s}=\sum_{i=1}^{p^s}\begin{pmatrix}p^s\\i\end{pmatrix}\zeta^{p^s-i}b^{i}.
 \]
 Assuming $\zeta\neq 0$, this valuation 
 can be bounded below by minimizing the valuation
 \[v_p\left(\begin{pmatrix}p^n\\i\end{pmatrix}\zeta^{p^n-i}b^j\right) =v_p\left(\frac{p^n}{i} ~b^i\right)\geq  n-v_p(i)+i/e\]
  for $1\leq i\leq p^{n}.$ 
  We may assume $i$ is a power of $p$, say $i=p^{\ell}$ with $\ell$ between $0$ and $n$. 
  If we treat $\ell$ as a continuous variable on the interval $0\leq \ell\leq n$, we find the valuation is bounded below by $n+c_e$ where 
\begin{equation}\label{EqRamConst}
c_e= \begin{cases} 1/e & \text{ if } e\leq \log(p)\\
 \frac{-\log(e)+1+\log\log (p)}{\log(p)} & \text{ if } \log(p)\leq e. 
\end{cases}
 \end{equation}
 If $e$ is large and $n$ small ($e\geq p^n \log(p)$), then the valuation minimizes at $\ell=s$, so replacing $c_e$ with $-n+p^n/e$ improves the bound. However, in the interest of giving a bound which is uniform in $n$ we ignore this potential improvement for small $n$.  
\end{proof}

Since we define $H_0(N;K_\frakp)$ in terms of the isotypical components $H_0^^g(N;K_\frakp)$
we will prove Theorem \ref{ThmCorrespondence} (\ref{ThmCorr_Principal}),(\ref{ThmCorr_AlphaCoeff}), and (\ref{ThmCorr_HeckeEqui}) before Theorem \ref{ThmPhmf1}.

\begin{proof}[Proof of Theorem \ref{ThmCorrespondence} (\ref{ThmCorr_Principal}),(\ref{ThmCorr_AlphaCoeff}), and (\ref{ThmCorr_HeckeEqui})]

Part (\ref{ThmCorr_Principal}) follows immediately from the construction.

For part (\ref{ThmCorr_AlphaCoeff}), suppose $F^\frakp\in H^g_0(N;K,\frakp)$. We have that 
\[
\D F^\frakp\subseteq M^!_2(N;K_\frakp)
\]
by the second regularized convergence. 
Since $M^!_2(N;K_\frakp)$ has a rational basis and the principal parts of $F^\frakp$ at cusps are defined over $K$, we may decompose $\D F^\frakp=f^\frakp+h^\frakp$ where $f^\frakp$ has coefficients in $K$ and $h^\frakp$ is a cusp form.
 Part (\ref{ThmCorr_AlphaCoeff}) is equivalent to the assertion that we may take $h^\frakp$ to be a multiple of $g^\frakp$. This follows by noting that for any positive integer $n$, we have that 
 $F^\frakp|_0(\hatT_n-a_g(n))$ is weakly holomorphic, and so $(\D F^\frakp)|_2(\hatT_n-a_g(n))$ has coefficients in $K$. In particular, if we decompose $h^\frakp$ into Hecke eigenforms, only the component corresponding to $g^\frakp$ can be transcendental.

Part (\ref{ThmCorr_HeckeEqui}) follows immediately from the  second regularized convergence, since the operators $A_n$ used in the construction commute with the Hecke algebra and Atkin--Lehner involutions.
\end{proof}

\begin{proof}[Proof of Theorem \ref{ThmPhmf1}]
Parts (\ref{ThmFn_Inclusion})-(\ref{ThmFn_qexp}) follow from the construction as explicit consequences of the second regularized convergence, and from the parts of Theorem \ref{ThmCorrespondence} proven above. Similarly, the existence and modularity of the derivative in part (\ref{ThmFn_Exact}) is an explicit consequence of the second regularization. That the co-kernel is the the space of holomorphic modular forms follows as in the Archimedean case. 
\end{proof}

\section{Structure of the space $H_0(N;K_\frakp)$}\label{SecIP}
  In this section we will discuss certain implications about the structure of $H_0(N;K_\frakp)$, paralleling the known structure of $H_0(N;\C).$ The main theorem of this section relate the Hecke algebra, notions of orthogonality and inner products, and the $p$-adic slopes of modular forms.
 
We have established that for both Archimedean and non-Archimedean places $\nu$, 
\[M^!_2(N;K_\nu)=M_2(N;K_\nu)\oplus \D H_0(N;K_\nu).\]
In the Archimedean case, the image $\D H_0(N;\C)$ is a distinguished subspace $S_2^{\perp}(N;\C)$ of $M_2^!(N;\C)$ of forms with vanishing constant terms which are orthogonal to the space of holomorphic cusp forms with respect to the regularized Petersson inner product.
  
Using Theorem \ref{PairingFormula}, the Bruinier--Funke pairing 
%
 defined in equation (\ref{DefPairing}) extends in an obvious way to our $\frakp$-adic harmonic Maass forms. If $F^\frakp\in H_{0}(N;K_\frakp)$ and $g\in S_k(N;K_\frakp)$ have $q$-expansions at cusps given by
 $$F|_kW_{D}(q)=\sum_{n}a_D(n)q^n \  \ \text{ and } \ \  g|_k W_D(q)=\sum_nb_D(n)q^n$$  
respectively, then define
  \[
\{g,F\}_\frakp= [\SL_2(\Z):\Gamma_0(N)]^{-1}
\sum_{D\mid N}\sum_{n\in \Z} a_D(-n)\cdot b_D(n).
\]
  Using the $p$-adic correspondence of Theorem \ref{ThmCorrespondence}, this gives a natural way to define an algebraic analog of the Petersson inner product $\langle ,\rangle_p$ for each newform, up to a choice of non-zero square norm for each newform. By linearity we may define 
\[
\langle \cdot,\cdot\rangle_\frakp: M^!_{2}(N;K_p)\times S_{2}(N;K_p)\to K_p,
\]
 where for simplicity, we may take the square-norm $\langle g,g\rangle_\frakp$ of any newform to be $1$. For questions of orthogonality, the explicit choice of norms is irrelevant. 
 For the sake of generality we will generally choose not to specify when the choice of square-norms matters. For instance, we can define the \emph{shadow} of a $\frakp$-adic harmonic Maass form in a way that depends on the $\frakp$-adic Petersson inner product.
 
 \begin{definition}
 Let $\langle \cdot,\cdot\rangle_\frakp$ be a choice of $\frakp$-adic Petersson inner product, let $F\in H_0(N;K_\frakp),$ and let $B$ be an orthogonal basis for $S_{2}(N;K_\frakp).$ Then define the \emph{shadow} of $F$, denoted by $\xi F$, in terms of $\langle \cdot,\cdot\rangle_\frakp$ by
 \[
 \xi F:= 
 \sum_{g\in B}\frac{\{ g,F \}_\frakp}{\langle g,g\rangle_\frakp}g.
 \]
 \end{definition}
 
  Interestingly, we can nearly characterize the orthogonal subspace in terms of the $\frakp$-adic \emph{slopes} of the $q$-expansions.
Suppose $F\in M^!_2(N;K)$ has a $q$-expansion at each cusp $\rho$ given by 
$\displaystyle\sum_{n\in\Z} a_\rho(n)q^n.$ 
Then the $p$-slope of $F$ at the cusp $\rho$ is defined by the limit of the $q$-series congruences
\begin{equation}\label{DefSlope}
\liminf_{n\to \infty} \frac{v_p(F|_2W_pU_{p^n}(q))}{n}.
\end{equation}
For instance if $g\in S_2(N;\Q)$ and $(p,N)=1$, then the Weil bound implies its slope is either $0$ or $1/2$. In the latter case, if $p\geq5$ then $a(p^n)$ must vanish for $n$ odd and $a(p^{2n})=(-p)^n.$ 

The orthogonality of the spaces $S_2^{\perp}(N;K_p)$ and the holomorphic cusp forms $S_2(N;K_p)$ follows directly from a duality relation between the coefficients of forms in a canonical basis for each of the spaces $S_2^{\perp}(N;K_p)$ and $H_0(N;K_p).$ This duality is similar to dualities studied by Zagier\cite{ZagierTraces} and Duke--Jenkins\cite{DukeJenkins}, and is closely tied to the action of the Hecke algebra on these spaces.

For $\delta\mid N$ and $m\geq 0$, we find that there exist functions $F^\frakp_{0,N;\delta,m}\in H_0(N;K_\frakp)$ defined uniquely by their principal parts 
\[
F^\frakp_{0,N;\delta,m}|_0W_D(q)=\begin{cases} q^{-m}+O(q) &\text{ if } D=\delta\\
O(1) &\text{ otherwise}.
\end{cases}
\]
For $m\neq 0$, let $F^\frakp_{2,N;\delta,m}:= -\frac {1}{m}\D F^\frakp_{0,N;\delta,m}$, and let the $q$-expansions at cusps of these functions be given by 
\begin{align*}
F^\frakp_{0,N;\delta,m}|_0W_D(q)&=\sum_{n\in \Z} A_{0,N}^\frakp(m,n;\delta,D)q^n,\\
F^\frakp_{2,N;\delta,m}|_2W_D(q)&=\sum_{n\in \Z} A_{2,N}^\frakp(m,n;\delta,D)q^n.
\end{align*}
Then we have the following theorem relating the Hecke algebra, duality, orthogonality, and slopes.

\begin{theorem}\label{ThmIP}
Assume the notation above. 
Then the following are true.
\begin{enumerate}
\item The space $H_{0}(N;K_\frakp)$ is generated by the action of 
$\T^*_{2-k}(N;K_\frakp)$ 
 acting on a single element. \label{ThmIP_Hecke}
\item The forms $F^\frakp_{0,N;D,m}\in H_0(N;K_\frakp)$ and $F^\frakp_{2,N;D,m}\in M^!_2(N;K_\frakp)$ with principal parts described above exist. Moreover, their coefficients are dual: if $m,n> 0$, then they satisfy the relations  
\[
nA_{0,N}^\frakp(m,n;\delta,D)=-mA_{2,N}^\frakp(m,n;\delta,D)
\]
and
\[
A_{0,N}^\frakp(m,n;\delta, D)=-A_{2,N}^\frakp(n,m;D,\delta).
\]\label{IP_dual}
\item The space $S_2^{\perp}(N;K_\frakp):=\D H_0(N;K_\frakp) \subset M^!_2(N;K_\frakp)$ 
 is the unique maximal subspace consisting of forms with vanishing constant terms which are orthogonal to the holomorphic cusp forms $S_2(N;K_\frakp)$ with respect to any normalization of the inner product $\langle,\rangle_\frakp$ described above. \label{IP_orth}
\item Let ${f\in M^!_2(N;K_p)}$. If $f\in S_2^{\perp}(N;K_\frakp),$ then it has slope at least $1/2$ at every cusp. Otherwise  
there is some cusp at which it has slope no greater than $1/2.$
\label{IP_slope}
\end{enumerate}
\end{theorem}
The description given in the theorem for $S_2^{\perp}(N;K_p)$ in terms of slopes is useful because it tells us when forms in $H_0(N;K_\frakp)$ can or cannot be standard $p$-adic modular forms. 
If $f\in S_2^{\perp}(N;K_\frakp)$ has slope at least $1$ at every cusp, then it has an anti-derivative $F\in H_0(N;K_p)$ with $f=\D F.$ Then $F$ has non-negative slope. It is a $p$-adic modular form of a type studied by Bringmann--Guerzhoy--Kane \cite{MockAsPAdic} and Kane--Waldherr \cite{KaneCong}. Moreover, $F$ is $p$-harmonic under Candelori's definition \cite{Candelori}. As we will see, this occurs if and only if the shadow of $F$ consists of only $p$-ordinary eigenforms. Candelori explicitly restricts his attention to this situation, although he phrases the condition in terms of the de Rham cohomology rather than the pairing. The statements are equivalent. 

In the alternative case, when $\{g,F\}\neq 0$ for some non $p$-ordinary eigenform, then it has negative slope. This violates the usual $q$-expansion principle for $p$-adic modular forms, and therefore $F$ is not a $p$-adic modular form in the usual sense.

\begin{proof}[Proof of Theorem \ref{ThmIP} part (\ref{ThmIP_Hecke})]
The proof follows as in the proof of Proposition \ref{PropCGen}. We need only prove the existence of a form $P_*^\frakp$ with the required principal part as in that proposition. If $\sigma$ is any Archimedean place of $K$, then the original complex function $P_*\in H_0(N;\C)$ can be decomposed as
\[
P_*=F+\sum_{\substack{g \in \mathcal S_M\\\delta\mid N/M}}F_{g,\delta},
\]
where  $F_{g,\delta}\in H^g(N;K;\sigma)$, and $F$ is weakly holomorphic with coefficients in $K$. We simply replace each form with its corresponding $\frakp$-adic form to obtain 
\[
P_*^\frakp:=F^\frakp+\sum_{\substack{g \in \mathcal S_M\\M|N\\\delta\mid N/M}}F_{g,\delta}^\frakp.
\]
\end{proof}
\begin{proof}[Proof of Theorem \ref{ThmIP} part (\ref{IP_dual})]
Given $P_*^\frakp$ defined above, have $F^\frakp_{0,N;1,1}=P_*^\frakp-c$ for some $c$. In general
\begin{equation}\label{BasisHecke}
F^\frakp_{0,N;\delta,m}=F^\frakp_{0,N;1,1}|_0\hatT_mW_\delta.
\end{equation}

The first relation for the coefficients follows immediately from the definition of $F^\frakp_{2,N;\delta,m}$.
Combined with the first, the second identity is equivalent to 
\[
nA_{0,N}^\frakp(m,n;\delta,D)=mA_{0,N}^\frakp(n,m;D,\delta).
\]
We prove this identity in parts. Let $\widetilde D=\frac{\delta D}{(\delta,D)^2}.$ Factor $n=n_1n_2n_3$, where $n_1$ is the largest factor of $n$ with $(n_1,N)=1$, and $n_2$ is the largest factor of $\frac{n}{n_1}$ with $(n_2,\widetilde D)=1.$ Factor $m=m_1m_2m_3$ defined similarly.

Using (\ref{BasisHecke}), we find 
$A_{0,N}^\frakp(m,n;\delta, D)=A_{0,N}^\frakp(m,n;1,\widetilde D).$ Formula (\ref{EqnTHatCoP}) allows us to work out the action of the operators $\hatT_{m_1}$ and $\hatT_{n_1}$ on  $F^\frakp_{0,N;1,\frac{m}{m_1}}$. Noting that $(m_1,n)=(n_1,m_1n_2n_3),$
 we find 
\begin{align*}
nA_{0,N}^\frakp(m,n;1,\widetilde D)&=\sum_{r\mid (m_1,n)}\frac{nm_1}{r}A_{0,N}^\frakp\left(\frac{m}{m_1},\frac{nm_1}{r^2} ;1,\widetilde D\right)\\
&=m_1n_2n_3 \ A_{0,N}^\frakp(n_1 \, m_2m_3,m_1 \, n_2n_3;1,\widetilde D).
\end{align*}
We also have that $\hatT_{m_3}W_{\widetilde D}=W_{\widetilde D} m_3 U_{m_3},$ which gives
\begin{align*}
m_1n_2n_3A_{0,N}^\frakp(n_1m_2m_3,m_1n_2n_3;1,\widetilde D)
&=m_1n_2n_3m_3A_{0,N}^\frakp(n_1m_2,m_1n_2n_3m_3;1,\widetilde D)\\
&=m_1n_2m_3A_{0,N}^\frakp(n_1m_2n_3,m_1n_2m_3;1,\widetilde D).
\end{align*}

At this point we have successfully exchanged $m_1$ and $m_3$ with $n_1$ and $n_3$ respectively. For $n_2$ and $m_2$ we consider prime factors individually. Suppose $\ell$ is a prime dividing $N/\widetilde D$ and $m'$ and $n'$ are positive integers coprime to $\ell.$ Then using equation (\ref{EqnTHat}) for the action of the operators $\hatT_{\ell^a},$ we find that 
\begin{align*}
\ell^{b}A_{0,N}^\frakp&(m'\ell^{a},n'\ell^{b};1,\widetilde D)\\
&=\sum_{i=0}^{\min(a,b)}
\ell^{b}A_{0,N}^\frakp(m',n'\ell^{a+b-2i};1,\widetilde D)\ell^{a-i}
-\sum_{i=1}^{\min(a,b)}
\ell^{b}A_{0,N}^\frakp(m',n'\ell^{a+b+2-2i};1,\widetilde D)\ell^{a+2-i}\\
&=\ell^{a}A_{0,N}^\frakp(m'\ell^{b},n'\ell^{a};1,\widetilde D).
\end{align*}
Applying this calculation for each prime $\ell\mid m_nn_2,$ we find that 
\[
nA_{0,N}^\frakp(m,n;1,\widetilde D)=mA_{0,N}^\frakp(n,m;1, \widetilde D),
\]
or equivalently
\[
nA_{0,N}^\frakp(m,n;\delta,D)=mA_{0,N}^\frakp(n,m;D,\delta).
\]
\end{proof}

\begin{proof}[Proof of Theorem \ref{ThmIP} part (\ref{IP_orth})] The forms $F^\frakp_{2,N;D,m}$ give an obvious basis in terms of principal parts which uniquely identify forms in the space. 

Orthogonality is an easy consequence of the duality relations. The paring is clearly bilinear. Since the functions $F^\frakp_{0,N;D,m}$ and $F^\frakp_{2,N;D,m}$ span $H_0(N;K_\frakp)$ and $S^\perp_0(N;K_\frakp)$ respectively, the statement follows from the equation  
\[
\{F^\frakp_{2,N;D,n}, F^\frakp_{0,N;\delta,m} \}=A_{0,N}^\frakp(m,n;\delta, D)+A_{2,N}^\frakp(n,m;D,\delta)=0.
\]
 for all positive integers $m,n$ and divisors $\delta,D$ of $N$. 
\end{proof}
\begin{proof}[Proof of Theorem \ref{ThmIP} part (\ref{IP_slope})]
The space of holomorphic modular forms $M_2(N;K_\frakp)$ has a basis of forms which are eigenforms for the Hecke operators $T_n$ for $(n,N)=1$ and for the $U_\ell$ operators for $\ell \mid N$.
 The sum of two forms with different slopes takes on the lesser of the two slopes. Similarly, the non-zero sum of eigenforms with the same slope must keep that same slope.

 The newforms of level divisible by $p$ have eigenvalue $\pm1$ with respect to the $U_p$ operator and hence slope $0$. If $g$ is new of level $M$ with $(M,p)=1$, then the regularized forms $g|_2W_\delta-\overline \beta^{-1}g|_2W_{p\delta}$ and $g|_2W_\delta-\beta^{-1}g|_2W_{p\delta}$  for $\delta\mid \frac{N}{pM}$ all have eigenvalues $\beta$ and $\overline \beta$ respectively. Although the latter has slope greater than $1/2$ at infinity,  a short calculation shows it has the same slope as $g$, i.e. $v_p(\beta)\leq 1/2$ at the cusp $0$. Similarly, every non-zero Eisenstein series has at least one cusp with slope $0$. Thus, every holomorphic modular form of weight $2$ has slope no greater than $1/2$ at some cusp.

The elements of $H_0(n;K)$ have slope at least $-1/2$, so their derivatives all have slope at least $1/2$.  Suppose $F^\frakp\in H^g_0(N;K,\frakp)$, and without loss of generality, assume $F^\frakp$ has $\frakp$-integral principal parts at all cusps. If $\frakp$ divides $N$, then $F^\frakp$ is constructed as in \ref{SubSecConstPN} as the limit of forms with $\frakp$-integral coefficient at all cusps. Thus $F^\frakp$ and $\D F^\frakp$ have slopes at least $0$ and $1$ respectively.  

If $\frakp$ does not divide $N,$ let $H^\frakp$ and $F_n$ be the functions defined as in section \ref{SecConstruction}. Recall $H^\frakp$ has $\frakp$-integral coefficients at all cusps. Since the coefficients of the functions $F_{n}$ converge $\frakp$-adically to those of $F^\frakp,$ we can use (\ref{DefBn}) to write the coefficients of $F_{n}$, and therefore the coefficients of $F^\frakp$, in terms of those of $H^\frakp.$
Suppose the $q$-expansions of these functions are given by 
\begin{align*}
H^\frakp(q)&=\sum_{m}c(m)q^m\\
F_{n}(q)&=\sum_{m}b_{n}(m)q^m\\
F^\frakp(q)&=\sum_{m}b_\infty(m)q^m.
\end{align*}

Note that the contribution of the second sum in (\ref{DefBn}) to any given coefficient becomes increasingly insignificant. If $m$ is coprime to $p$ and $v_p(\beta)<1/2$, then, using (\ref{EqnTHatCoP}) we find that
   \begin{equation}\label{CoeffsLimitFormula}
   \begin{split}
  b_\infty(mp^s)=\lim_{n\to \infty}b_{n}(mp^s)&=\sum_{j=0}^\infty \beta^{-1-j}\sum_{i=0}^{\min(j,s)}p^{j-i} c(mp^{s+j-2i})\\
  &=\sum_{h=0}^\infty c(mp^{h}) \sum_{i=\max(s-h,0)}^{s} \beta^{-1-i}\overline\beta^{h-s+i}.
  \end{split}
  \end{equation}
  The minimum valuation from the inner sum comes from the $(h,i)=(0,s)$ term, so we have slope at least $-v_p(\beta)$ at infinity. An identical argument works at the other cusps.
  When $v_p(\beta)=1/2,$ we must add a second term similar to that above swapping $\beta$ and $\overline \beta$, however both terms have a slope of at least $-1/2.$ 
  
  If $p$ is greater than $3$ and is unramified in $K$, then the sum becomes more interesting. In this case $\beta= -\overline\beta=\pm \sqrt{-p}$, and we have
   \begin{equation}\label{CoeffLimitFormula2}
   \begin{split}
  b_\infty(Ap^s)
  &=\sum_{h=0}^\infty c(Ap^{h}) \sum_{i=\max(s-h,0)}^{s} (\beta^{-1-i}\overline\beta^{h-s+i}+\overline\beta^{-1-i}\beta^{h-s+i})\\
&=\sum_{h=0}^\infty c(Ap^{h}) ( \beta^{-1}\overline\beta^{h-s}+\overline\beta^{-1}\beta^{h-s})\sum_{i=\max(s-h,0)}^{s} (-1)^i.
  \end{split}
  \end{equation}
  The outer sum vanishes if $h$ and $s$ have the same parity, but the inner sum vanishes if $h$ is odd and $h\leq s$. Since denominators only arise from the terms with $h\leq s,$ the coefficients 
  $b(Ap^{2n})$ are all integral. If $p\mid 6$ is unramified, then we have that $\beta\equiv -\overline\beta \pmod{p}$ and so similarly the denominator vanishes.
  \end{proof}
\subsection{Slopes and $p$-adic modular forms}\label{SecSlopes}
The $q$-expansions can also be studied using techniques of Guerzhoy and Guerzhoy--Kent--Ono. Given a newform $g$, let 
\[\mathcal E(q):=\sum_{n\geq 1}\frac{a_g(n)}{n}q^n\]
be the \emph{Eichler integral}, or formal antiderivative, of $g(q).$

If $F^\sigma\in H^g_0(N;K,\sigma)$ so that the holomorphic part $F_g^{\sigma+}$ has $q$-expansion $\sum_{n}b^\sigma(n)q^n,$ then as seen previously, if $\alpha\in K+b (1),$ we have that 
 \[
F_\alpha:= F^{\sigma+}(q)-\alpha \mathcal E^\sigma(q)\in K((q)).
 \]
This can be improved to address $\frakp$-denominators for any fixed prime $\frakp$. Proposition 5 of \cite{ZagiersAdele} (see also \cite{MockAsPAdic,PAdicCoupling}) shows that if $\alpha^\sigma$ is chosen as above, then for each prime $\frakp,$ there are $\frakp$-adic constants $\lambda_\frakp(\alpha), \mu_\frakp\in K_\frakp$  so that  the $q$-series 
 \[
\widetilde F_{\frakp}(q):=F_\alpha-\lambda_\frakp \mathcal E(q)-\mu_\frakp \mathcal E(q^p),
 \] 
 has bounded $\frakp$-denominators, and is independent of the choice of $\alpha$. 
These constants are unique for a given $\alpha$ if $g$ is not $p$-ordinary. Otherwise, $\mathcal E(q)-\frac{\beta}{p} \mathcal E(q^p)$ has integral coefficients, but there is a distinguished choice of $\widetilde F_{\frakp}(q)$ so that  $\mu_\frakp=0.$

The operator $(a_g(p)-\hatT_p)$ annihilates $\alpha \mathcal E(q)$, and so we can write $F^\frakp$ in terms of $\widetilde F_{\frakp}(q),$ $\mathcal E(q)$, and $\mathcal E(q^p).$ Equation (\ref{prodH}) shows that in the limit, the operators $(a_g(p)-\hatT_p)B_{n}$ preserve $\widetilde F_{\frakp}(q),$ but will turn $\mathcal E(q^p)$ into some combination of $\mathcal E(q)$ and $\mathcal E(q^p).$ If $v_p(\beta)<1/2$, then by considering the slope, we find that  
\[
F^\frakp_g(q)=\widetilde F_{\frakp}(q)+C \cdot \left(\mathcal E(q)-\frac\beta p \mathcal E(q^p)\right)
\]
 for some constant $C\in K_\frakp$. When $\beta=-\overline \beta$ 
 , then there is some $C\in K^\frakp$ so that 
 \[
F^\frakp_g(q)=\widetilde F_{\frakp}(q)+C\cdot \mathcal E(q^p).
\]

\section{Integrality}\label{SecIntegrality}

Here we prove the bounds on denominators that may arise in the $q$-series and the evaluations of the $\frakp$-adic harmonic Maass forms. Throughout this section let $(F^\nu)_\nu$ be a family of functions as in Theorem \ref{ThmIntegrality}, let $K_N$ be the smallest number field containing the coefficients of every newform of level dividing $N$, and let $\O_\frakp$ be the valuation ring of $K_\frakp$. The constants $\mathcal M_N$, $\mathcal R_N$, and $\mathcal B_N$ used in the theorem will be defined explicitly in the course of the proofs.

 \begin{proof}[Proof of Theorem \ref{ThmIntegrality} (\ref{ThmInt_coeff})] 
 The first question is similar to one considered by Guerzhoy in \cite{ZagiersAdele}. Much of the work here, as in section \ref{SecSlopes}, parallels work done in that paper. 
 
For simplicity at first, suppose $F^\nu\in H^g_0(N;K,\nu)$ for some newform $g$ of level $N$.
For each prime $\frakp$ of $K$, let the $q$-series $\widetilde F_{\frakp}(q)$ and $\mathcal E(q),$ and the $\frakp$-adic constants $\lambda_\frakp$ and $\mu_\frakp$ be defined as in the previous section for some appropriate choice of $\alpha^\frakp$. Then Theorem 1 of \cite{ZagiersAdele} asserts that for all but at most finitely many primes $\frakp$ of $K$, the numbers $\lambda_\frakp$ and $p\mu_\frakp$ are $\frakp$-integral. This is proven using a formula similar to equation (\ref{CoeffLimitFormula2}). In our case, the equation shows that if $\frakp$ does not divide $N$ and $(m,p)=1,$ then 
\[
\beta^{n+1}\cdot a^\frakp_{\delta}(mp^n) \in \O_{\frakp} (\beta).
\]
If $p$ is not ramified in $K_N$, then we cannot have fractional $p$-valuations and so
 \[st \cdot a^\frakp_{\delta}(s^2t)\in \O_\frakp.\]
 Otherwise we may have 
\[
v_\frakp\left(\sqrt{m \mathcal R_N} \cdot a^\frakp_{\delta}(m)\right)\geq 0,\]
where $\mathcal R_N$ is the radical of the norm of $K_N$, which contains a single power of the rational primes which ramify in $K_N.$

If $\frakp$ divides the level of $g$, then the construction can be made by taking the limit of forms with $\frakp$-integral coefficients, and so no denominators arise in this case.

There is an additional source of potential denominators when $F^\nu\not \in H^g_0(N;K;\nu)$ for any newform $g$. The additional possible denominators can be recovered by a linear algebra argument as follows. 

For each prime $\frakp$ of $K$, let $F^\frakp_{0,N;1,1}$ be the function described in Section \ref{SecIP}. 
 with principal part $q^{-1}+O(q)$ at infinity and no other singularities. Then $\D F_{0,N;1,1}^\frakp\in M^!_2(N;K_\frakp)$ and so it has a bound, say $s_\frakp$, on the $\frakp$-valuation of its denominators. Since $F_1^\frakp$ generates the space (besides perhaps the constant functions) under $\TT^*_0(N;K_\frakp),$ we have that this bound must hold for every form $\D F^\frakp$ with ${F^\frakp\in H_0(N;K_\frakp)}$ and with principal parts defined over $\O_\frakp$. 

For each newform $g$ of level $M$ dividing $N$ and each divisor $\delta$ of $N$, we have the pairing
\[
\{g|_2W_\delta, F_{0,N;1,1}^\frakp\}=\begin{cases}\lambda_{g,\delta}\mathcal I^{-1} & \text{ if }\delta\mid M\\ 0 & \text{ if }\delta \not | M, \end{cases}
\] 
where $\mathcal I:=[\SL_2(\Z):\Gamma_0(N)]$ and $\lambda_{g,\delta}=\pm 1$ is the eigenvalue of $g$ under $W_\delta$.  

For each such $g$, suppose we have a fixed function $F_g^\frakp\in H_0(M;K,\frakp)$ and non-zero constant $C_g\in K_\frakp$ 
which satisfies 
\[
\{h|_2W_\delta, F_g^\frakp\}=\begin{cases} \lambda_{g,\delta}\mathcal I^{-1}C_g & \text{ if }h=g \text{ and } \delta\mid M\\ 0 & \text{ otherwise } \end{cases}
\]
for each newform $h$ of level dividing $N$, and each divisor $\delta$ of $N.$ Without loss of generality, we may assume each $F_g^\frakp$ has integral principal parts at all cusps. In this case, the formula for the pairing involves only integral coefficients so $C_g\in \O_K$.
Moreover, the function 
 \[
G:= F_{0,N;1,1}^\frakp-\sum_{\substack{g \in\mathcal S_M\\M| N}} \frac{1}{C_g}F^\frakp_g
 \]
 is weakly holomorphic since its pairing with any cusp form is $0$. Its principal part has denominators  at worst the least common multiple of the $C_g,$ and so this is true for all its coefficients. Thus we would like to minimize the norm of $C_g$ for each such form $g$. In subsection \ref{SubsecMinShadow} we show how to construct functions $F_g^\nu$ satisfying the conditions above and then compute the minimum bound $\mathcal M_N$. This number is closely connected to congruences between orthogonal cusp forms.  For instance, suppose $g,$ $f_g$ and $C_g$ are as above, and suppose there is some form $g'\in S_{2}(N;K)$ with coefficients in $\O_K$ which is orthogonal to $g$, but satisfies $g(q)\equiv g'(q) \pmod{\mathcal M}.$ Then using the formula for the pairing, we must have that  
 \[
 C_g=\mathcal I\cdot \{g,F_g^\frakp\}\equiv_{(\mathcal M)} \mathcal I\cdot \{g',F_g^\frakp\}=0.
 \]
 Thus $\mathcal M \mid C_g$ which implies $\mathcal M\mid \mathcal M_N.$
 
With $\mathcal M_N$ defined as claimed, then altogether we see that if $F^\nu$ has principal parts in $\O_K$ at all cusps, then denominators of the coefficients $a^\nu_d(n)$ can only arise from:
\begin{enumerate}
\item The index, $n$, of the coefficient,
\item Ramified primes,
\item The congruence number $\mathcal M_N.$
\end{enumerate}
In particular, we recover the statement of part (\ref{ThmInt_coeff}) of theorem \ref{ThmIntegrality}.
\end{proof}
 \begin{proof}[Proof of Theorem \ref{ThmIntegrality} (\ref{ThmInt_val})] 
 As before, assume that $F^\nu\in H^g_0(N;K;\nu)$ for some newform $g$ of level $N$.  If $\frakp$ divides the level of $g$, then the function can be taken as the limit of functions with $\frakp$-integral coefficients, so no denominators arise. 
  
 Suppose $\frakp$ does not divide the level, and let $F_n$ be the functions defined in section \ref{SubSecConstruction}. We begin in the case $v_\frakp(\beta)>0.$ Using equation (\ref{ValDiff}), we see that unless $p$ is severely ramified in $K$,  we have the congruence of values, $F_n(E)\equiv F_1(E)\pmod{1}.$ Thus we need only work out denominators for $F_1.$ Using equation (\ref{prodH}), we see that $\beta F_1(q)$ is $\frakp$-integral, and so $\beta F_{n}(E)$ must be. 
 
 If $p$ does not ramify but $v_p(\beta)>0,$ then $\beta=-\overline \beta$ (or if $\frakp$ divides $6$, we have $\beta\equiv-\overline \beta\pmod \frakp$). In this case, using the second case of definition \ref{Def_Fp}, we see the single power of $\beta$ in the denominators of $F_1$ and $\overline F_1$ must cancel to eliminate the fractional power of $p$, and we find that $F(E)$ is $\frakp$-integral.
 
 If $p$ has high ramification degree, (i.e. $e\geq \log p$), it may turn out that $1+c_e<0$ . However from the proof of Proposition \ref{HeckeLimit}, we see that at worst we can replace $n+c_e$ with $0$, and we have that $\beta^2F_n(E)\equiv \beta^2 F_1(E)\pmod 1.$ Since $\beta^2 F_1(q)$ has $\frakp$-integral coefficients, $\beta^2 F_1(E)$ (and hence $\beta^2 F_n(E)$) has positive $\frakp$-adic valuation. 
 
If $p$ does not divide $N$ and $v_p(\beta)=0,$ we have potential denominators arising from the $1-\beta^{-d}$ terms which appear in equation (\ref{DefBn}). These can be cleared by multiplying by 
\begin{equation}\label{B_p,g}
\mathcal B_{\frakp,g}:=1-\beta w_p^{-1}(\beta).
\end{equation}

\noindent If $N':=\frac{N}{\gcd(N,p)}$, then 
 define $\mathcal B_{N,p}=p^r$ where
\begin{equation}\label{B_p}
r:=\max \{v_p(1-\beta_hw_p^{-1}(\beta_h) \ : \ h \in \mathcal S^{\text{new}}(N')\}.
\end{equation}

\noindent Then if 
 $j(E)$ is $p$-integral, we must have that 
\begin{equation}\label{EqnValBoundIso}
v_\frakp\left(\mathcal R_N  F^\frakp(E) )\right)\geq -\mathcal B_{N,p}.
\end{equation}
If we remove the extra assumption that $F^\frakp\in H^g_0(N;K;\frakp)$ for some newform $g$, then as in the previous proof we allow additional denominators dividing $\mathcal M_N.$

\end{proof}
\subsection{Calculating $\mathcal M_N$} \label{SubsecMinShadow}
The calculation is a linear algebra problem. Let $\dim_N=\dim S_2(N;K).$
There is a basis $S_2(N;K)$ has a basis of cusp forms $h_i$ in reduced echelon form so that their $q$-expansions have the shape 
\[h_i(q)=q^{e_i}+\sum_{\substack{j=i+1\\j\neq e_1,e_2,\dots}}^\infty a_{i,j} q^{j}
\] with 
\[1=e_1<e_2<\dots <e_{\dim_N}.\]
The coefficients $a_{i,j}$ of each $h_i$ are in $\Q.$

Since $N$ is square free, $S_2(N;\C)$ has another basis $\mathcal S_N$ given by 
\[\mathcal S_N=\{ h|_2W_D \ : \ h \text{ is a normalized new form of some level } M\mid N, \text{and } D\mid \tfrac{N}{M} 
\}\] 
%
 \noindent
Let $\mathbf C$ be the $\dim_N\times\dim_N$ matrix whose rows are indexed by $\mathcal S$ and whose $(h,i)$-th entry is the $e_i$-th coefficient of $h$ for $h\in \mathcal S.$ The entries of $\mathbf C$ are algebraic integers in some extension $L/\Q$, but the entries of $\mathbf C^{-1}$ my not be integral. Define  $\mathcal M_N$ to be the least positive integer so that the entries of $\mathcal M_N \mathbf C^{-1}$ are integral. 
 It is not hard to see that the modulus of any congruences among forms in $\mathcal S_N$ must divide $\mathcal M_N.$ Additionally, the $i$-th row of $\mathbf C^{-1}$ gives the coefficients to write the form $h_i$ in terms of the forms in $\mathcal S_N$, so any denominators of the coefficients $a_{i,j}$ must also divide $\mathcal M_N$. 

The columns of $\mathbf C^{-1}$ encode the principal parts at infinity for harmonic Mass forms corresponding to each $h\in \mathcal S_N.$ For each $h\in \mathcal S_N$, let $G_h$ be the unique harmonic Maass function whose principal part at infinity is given by 
\[
G_h(q)=\sum_{i=1}^{\dim_N} c_{h,i} q^{-e_i} +O(q)
\]
where $c_{h,i}$ is the $h,i$-th entry of $\mathbf C^{-1}$, and which is bounded at all other cusps. Calculating the Bruinier--Funke pairing, we find that for $g\in \mathcal S_N,$
\[
\{g,G_h\}=\begin{cases}1& \text{ if } g=h\\0&\text{otherwise.}
\end{cases}
\] 

Proceeding as above, we see that for a prime $\frakp$ of $K$, if $\mathcal S^{\text {new}}$ is the set of newforms of levels dividing $N$, then the form 
\[
G^*:= P_*^\frakp-\sum_{h \in \mathcal S^{\text {new}}}G_h^{\frakp}
\]
is weakly holomorphic, where $P_*^\frakp$ is the form defined in the proof of Theorem \ref{ThmIP} part (\ref{ThmIP_Hecke}). Since all denominators in the expansions of $G^*$ and the $G_h^*$ divide $\mathcal M_N,$ the same is true for any denominators of $P_*^\frakp.$

\section{Coefficients of half-integeral weight forms}\label{SecHalfWt}

The definition of harmonic Maass forms given in Section \ref{SecMaass} can be generalized to forms of half integral weight. Any such generalization requires that we allow a non-trivial multiplier system, such as in the modularity of the Jacobi theta function in order to maintain consistency. 
 
Suppose $\nu:\Gamma\to \C$ is a function on some subgroup $\Gamma\subset\GL_2(\R)^+$. 
We modify the definition of the slash operator given in (\ref{EqnMatrixAction}) with respect to $\nu$ so that 

\[
(f|_{k,\nu}\gamma)(\tau)=\nu^{-1}(\gamma) \ |\operatorname{Det}(\gamma)|^{k/2} \ (c\tau+d)^{-k}f(\gamma\tau)
\]
for every $\gamma=\sabcd\in \Gamma.$ Here, if $k$ is a half-integer, then the resulting square roots take the principal branch. Then $\nu$ is a weight $k$ multiplier system for $\Gamma$ if 
\[
(f|_{k,\nu}\gamma)|_{k,\nu}\mu=f|_{k,\nu}(\gamma\mu) \  \text{ for all } \gamma,\mu\in \Gamma,
\]
 independent of the function $f$.

This definition extends naturally to \emph{vector-valued} modular and harmonic Maass forms. These satisfy a similar definition, but now we allow $f$ to represent a vector of functions and require the image of $\nu$ to be unitary matrices.

Definition \ref{DefHMF} extends naturally to include modularity involving multiplier systems and vector valued functions. The remaining facts about harmonic Maass forms still hold (component wise in the case of vector valued forms), except that we may have fractional powers of $q$ in the $q$-series expansions, depending on the multiplier system. 

We are interested in lifts from weight $0$ forms to vector valued forms of weights $1/2$ and $3/2$, with certain explicit multiplier systems. These vector valued forms may be projected down to obtain scalar valued half-integer weight forms. For instance, the lifts considered implicitly by Zagier\cite{ZagierTraces}, and by Miller and Pixton \cite{MillerPixton} can be obtained by summing 
the components of the lifts considered here. 

\subsection{Vector valued modular forms and the Weil representation}
We refer the reader to \cite{AlfesKudlaMillson,ECMock,BruinierFunke} for a more thorough consideration of vector valued modular forms. Here we will give a summary aimed at considering forms with multiplier systems determined by the Weil representation. 

The metaplectic group $\operatorname{Mp}_2(\Z)$ consists of pairs $(\gamma,\phi)$ where $\gamma=\abcd\in \SL_2(\Z)$ and $\phi:\H\to \C$ is holomorphic and satisfies 
$$\phi^2(\tau)=c\tau+d,$$ with the group action given by 
\[
(\gamma_1,\phi_1)(\gamma_2,\phi_2)=\left(\gamma_1\gamma_2, \phi_1(\gamma_2\tau)\phi_2(\tau)\right).
\]
The metaplectic group is generated by the two elements 
\[T:=\left(\zxz{1}{1}{0}{1}, 1\right) \ \text{ and } \ S:=\left(\zxz{0}{-1}{1}{0}, \sqrt{\tau}\right ).\]

Given a positive integer $N$, the Weil representation $\rho$ of the metaplectic group acts on a vector space indexed by $\Z/2N\Z.$ This action is given for the generators $T$ and $S$ so that $\rho(T)$ is the diagonal matrix with $j$-th entry $e^{2\pi \I\frac{j^2}{4N}}$, and $\rho(S)$ is the square matrix with $(j,k)$-th entry given by $\frac{\sqrt{-\I}}{\sqrt{2N}}e^{-2\pi \I\frac{j\cdot  k}{2N}}$ for $j,k\in \Z/2N\Z.$ 

We then define the slash operator $ |_{k,\rho}$ so that if $\mathbf f$ is a vector valued function of  dimension $2N$, then 

\begin{equation}\label{vvTrans}
\mathbf{f}|_{k,\rho}\gamma(\tau):=\phi^{-2k}(\tau)\rho^{-1}(\gamma) \mathbf{f}(\gamma \tau).
\end{equation}
 
We say that $\mathbf f$ is modular with respect to $\rho$ with weight $k$ if $\mathbf f$ is invariant under the action of the slash operator given above. We will also consider modularity with respect to the conjugate representation $\overline \rho,$ and here and throughout we will use $\tilde \rho$ to mean either $\tilde \rho=\rho$ or $\overline \rho.$
We denote the $\C$ vector spaces of harmonic Maass forms of weight $k$, index $N$ for the representations $\rho$ and $\overline \rho$ by $H_{k, \rho}(N;\C)$ and $H_{k,\overline \rho}(N;\C)$ respectively.

These representations allow us to compute the transformation of such a form $\mathbf f$ with respect to any matrix $\gamma\in \SL_2(\Z).$ 
The action of the element $T$ (which does not depend on the weight $k$) gives additional information about the $q$-expansions of the components of such forms. 
If $\mathbf{f}$ is in $H_{k, \tilde \rho}(N;\C)$, then the holomorphic part of the $j$-th component, $\mathbf{f}_{j}^+$, will have a $q$-expansion of the form 
\[
\mathbf{f}_{j}^+=\sum_{\substack{D\equiv j^2 \pmod {4N}\\ \lambda D\gg-\infty}}a^+_{\mathbf f}(\lambda D,j)q^{\lambda\frac{D}{4N}},
\]
where $\lambda=1$ if $\tilde \rho=\rho$, and $\lambda =-1$ if $\tilde \rho=\overline\rho$. Given such an expansion, we will refer to $a^+_{\mathbf f}(D,j)$ as the $(D,j)$-th coefficient of $\mathbf f.$

\subsection{The Hecke algebra on half-integral weight modular forms}\label{SecHalfWtHecke}
Half-integer weight modular forms transforming under the Weil representation $\rho$ or under $\overline \rho$ have a theory of Hecke operators. 
As in the integer weight case, the Hecke operators may be defined by taking a trace over the action of matrix coset representatives, so that by construction they preserve modularity properties. 
However, for simplicity here we will only define the operators $\mathbf T_n$ by their action on $q$-series. This action can be given in terms of operators $\mathbf U_n$, $\mathbf V_n$ and $\mathbf S_n$ which act on $q$-series as follows. For further details on Hecke operators for vector valued modular forms of half integral weight see \cite[section 7]{BruinierOno}, \cite{BruinierStein} or \cite[\S 0 equation (5)]{SZ}. Suppose $\lambda=\pm1$ and $\mathbf f(q)$ is a $q$-series vector with components $\mathbf f_s$ indexed by $s\in \Z/(2N\Z)$ which has the $q$-series expansion 
\[\mathbf f_s =\sum_{\substack{D\equiv s^2 \ (4N)\\ \lambda D\gg -\infty}} a(\lambda D,s)q^{\lambda\frac{D}{4N}}.\]
Then we define the operators $\mathbf U_n$, $\mathbf S_n$, and $\mathbf V_n$ so that
the $s$-th components of $\mathbf f|\mathbf U_n$, $\mathbf f|\mathbf S_n$, and $\mathbf f|\mathbf V_n$ respectively are given by
\begin{align*}
\left(\mathbf f|\mathbf U_n\right)_s 
&=\sum_{D\equiv s^2 \ (4N)} a(\lambda D n^2,s n)q^{\lambda \frac{D}{4N}}
\\
\left(\mathbf f|\mathbf S_n\right)_s 
&=\sum_{D\equiv s^2 \ (4N)} \leg{ D}{n}a(\lambda D,s)q^{\lambda \frac{D}{4N}}\\
%
\left(\mathbf f|\mathbf V_n\right)_{s} 
&=
\sum_{\substack{D\equiv r^2 \ (4N)\\rn\equiv s \ (2N)}} a(\lambda D,r)q^{\lambda \frac{n^2\cdot D}{4N}}.
\end{align*}
For 
$k\geq3/2$
 and $n$  coprime to $N$, the Hecke operators $\mathbf T_{n}$ are defined by the action on $q$-series given by 
  \begin{equation}\label{EqnTHalf1}
  \mathbf T_{n}:= \sum_{\substack{ a,b,c\geq 1\\ a\cdot b\cdot c=n}} b^{k-3/2}c^{2k-2}
  \mathbf{U}_{a} \mathbf{S}_b\mathbf{V}_c.
  \end{equation}
 If 
 $k\leq 1/2$
 , define the normalized Hecke operators
  \begin{equation}\label{EqnTHalf2}
  \mathbf T_{n}:= \sum_{\substack{ a,b,c\geq 1\\ a\cdot b\cdot c=n}} a^{2-2k}b^{1/2-k}%
  \mathbf{U}_{a} \mathbf{S}_b\mathbf{V}_c.
  \end{equation}
The formula for unnormalized operators of prime index is given in \cite[eq. (7.1)]{BruinierOno}.
A short calculation shows that 
the weight $1/2$ and $3/2$  operators satisfy the same multiplicative relations as do the $\hatT_n$ operators of weight $0$ or $2$:
\[\mathbf T_n\mathbf T_m=\sum_{d\mid (m,n)} d\mathbf T_{\frac{mn}{d^2}}.
\]

We also have operators corresponding to the Atkin--Lehner involutions. Vector valued modular forms under $\rho$ or $\overline \rho$ must be symmetric in the indices, up to sign, so 
\[\mathbf f_s=\pm\mathbf f_{-s},\]
where the sign depends on the representation. This is due to the fact that the element
\newcommand\bigzero{\makebox(0,0){\text{\huge0}}}
\[
S^2=\left(\zxz{-1}{0}{0}{-1},\operatorname{i}\right)
\]
of the metaplectic group must act as the identity on the forms, but 
\[
\rho(S^2)=-\operatorname{i}\cdot \hat I_{2n} =-\overline {\rho}(S^2)
\]
where $\hat I_{2n}$ is the reflection of the identity matrix,
\[\hat I_{2n}:=
\left(
\begin{array}{ccccc}
0&&1\\
&\reflectbox{$\ddots$}&\\ %
1&&0
\end{array}
\right).
\]
 Equation (\ref{vvTrans}) then implies that weight $1/2$ forms and weight $3/2$ forms with the same representation must have opposite signs. As we will see, this symmetry in the indices is related to the Fricke involution  (see (\ref{FrickeCom})), and can be generalized to Atkin--Lehner involutions. For each $\delta\mid N,$ let $\lambda_\delta$ be the unique number $\pmod{2N}$ satisfying 
\begin{equation}\label{halfWtAL}
\lambda_\delta\equiv -1 \pmod {2\delta} \ \text{ and } \ \lambda_\delta\equiv 1 \pmod {2N/\delta}.
\end{equation}
 We define the Atkin--Lehner involution $\mathbf W_\delta$ for vector valued forms which acts by permuting the component indices by multiplication by $\lambda_\delta$, so that
 \[
(\mathbf f|\mathbf W_\delta)_s=\mathbf f_{\lambda_\delta s}.
 \]
It follows that forms which are modular with respect to $\rho$ or $\overline\rho$ are necessarily eigenfunctions for the Fricke involution $\mathbf W_N$.

We conclude this section with a brief discussion about the field of definition and denominators for coefficients of weight $1/2$ forms.

Similar to the weight $0$ case, the field of definition for the principal part of a form $\mathbf f \in M^!_{1/2,\tilde \rho}(N;\C)$ determines the field of definition for \emph{almost all} the coefficients.  
However there are possible exceptions for the representation $\tilde \rho=\rho$ and coefficients of square index. The Serre--Stark basis theorem implies  that the space of holomorphic modular forms $M_{1/2,\rho}(N;\C)$ may be non-empty (but is spanned by unary theta functions), whereas 
$M_{1/2, \overline \rho}(N;\C)$ is trivial. See \cite[Lemma 6.4]{BruinierOno}. 
\begin{theorem}[{\cite[Lemmas 6.3,6.4,6.5]{BruinierOno}}]\label{ThmSturmHalf}
Suppose $\mathbf f \in M^!_{1/2, \tilde \rho}(N;\C)$ has coefficients $a_{\mathbf f}(m,r)\in K$ for all $m\leq 0.$ Then $a_{\mathbf f}(n,s)\in K$ for all $n\in \Z$, unless $\tilde \rho=\rho$ and $n$ is a square.
If $a_{\mathbf f}(m,r)\in K$ for all $m<4N,$ then $a_{\mathbf f}(n,s)\in K$ for all $n.$
 \end{theorem} 
 
If $\mathbf f\in M^!_{1/2,\tilde \rho}$ has algebraic coefficients, then its coefficients naturally have bounded denominators. This fact will be important later on to show that differences of certain Hecke operators vanish $\frakp$-adically as in the integer weight case. It is less clear whether the algebraic coefficients of $\mathbf f$ have a bound on denominators when $\mathbf f\in H_{1/2,\tilde \rho}$ is not weakly holomorphic. This question is posed in remark 15 i) of \cite{BruinierOno}. While it appears to still be open whether or not a uniform bound exists for all algebraic coefficients, the following weaker lemma for a single square class suffices for our needs.

\begin{lemma}\label{LemSqrDenoms}
Suppose $\mathbf f\in M^!_{1/2,\tilde \rho}(N;\C)$ so that each component has principal part defined over $K$, and $\mathbf f$ is orthogonal to all cusp forms in $S_{1/2,\tilde \rho}(N;\C).$ If $D$ is a fixed fundamental discriminant with $s^2\equiv D\pmod {4N}$ so that the coefficients $a_{\mathbf f}(Dn^2,sn)$ are in $K$ for all $n\in \Z$, then the denominators of these coefficients are bounded. 
\end{lemma}

The condition that $\mathbf f$ is orthogonal to $S_{1/2,\tilde \rho}(N;\C)$ is only of concern when $D=1$ and $\tilde \rho=\rho.$ The proof of this lemma will require the use of the generalized Borcherds lift which we will discuss in the next section.

\section{Lifts of half-integer weight forms}\label{SecLifts}
Harmonic Maass forms of half integral weight are strongly connected to harmonic Maass forms of even integral weight. These connections include the Shimura correspondence, trace lifts, and the Borcherds lifts. 

The Shimura correspondence relates weight $2$ Hecke eigenforms to weight $3/2$ eigenforms for the Hecke operators and Atkin--Lehner involutions defined in Section \ref{SecHalfWtHecke} above, with the same eigenvalues as the weight $2$ form. 

Zagier~\cite{ZagierTraces} demonstrated that the coefficients of certain weight $1/2$ and $3/2$ modular forms can be given as sums of CM values of the $j$-function 
$\widetilde {\mathbf {t}}_1(\Delta,D,rs)(j).$ 
These results have been studied and generalized in several directions. Miller--Pixton~\cite{MillerPixton} extended these results to harmonic Maass forms of higher level and negative weight using formulas for coefficients of Maass--Poincar\'e series; Duke--Jenkins~\cite{DukeJenkins} demonstrated integrality results for similar traces related to modular forms of negative weight and level $1$; and Bruinier--Funke~\cite{BruinierFunke,BruinierFunkeTraces} and Alfes~\cite{AlfesKudlaMillson} have realized these trace maps as theta lifts obtained by taking the inner product of modular functions against certain non-holomorphic theta kernels. These and related theta lifts have been further studied by Alfes~\cite{ ECMock, AlfesMillson}, Duke--Imamo{\=g}lu--T\'oth~\cite{DIT1,DIT2}
and others~\cite{ThetaLiftings,BruinierOno2,Hovel}. From Zagier's work and the work of Duke--Jenkins, it is not hard to see that certain $p$-adic properties of $q$-series are propagated through the lifts. 
 Some of these properties have been explored by Bringmann--Guerzhoy--Kane \cite{PAdicCouplingHalf}. In the case of weakly holomorphic modular forms, the CM values which make up the traces are algebraic numbers, making inherent $p$-adic properties much easier to explore. When the original function is a harmonic Maass form, the traces are generally transcendental with exceptions involving the vanishing of the associated central $L$-values or $L$-derivatives (see \cite{BruinierOno}). 
 Exploring the $p$-adic properties arising from these transcendental numbers can be more problematic. However, the $\frakp$-adic harmonic Maass forms give us a tool to approach this question.

The Shimura correspondence and the three lifts we consider here all respect the Hecke algebra. In the case of the trace-lifts, this allows us to extend the lifts to our $p$-adic harmonic Maass forms by passing the respective operators and limits through the lift. Each of these lifts can be described in terms of theta lifts obtained by integrating a modular or harmonic Maass form against a two-variable theta kernel. These theta kernels are indexed by a fundamental discriminant, and are usually non-holomorphic but modular in each variable with different weights of modularity. 
Two families of \emph{trace-lifts} are defined 
by integrating against the Millson theta kernel (see \cite{ECMock,AlfesMillson}), and the Kudla-Millson theta kernel (see \cite{AlfesKudlaMillson,BruinierFunkeTraces}) respectively. The first of these lifts weight $0$ harmonic Maass forms to weight $1/2$ harmonic Maass forms, while the second lifts a weight $0$ forms to weight $3/2$ forms.  In both cases, the lifts give formulas for the coefficients of the holomorphic parts in terms of modular traces over CM points of the original modular function. These two lifts are dual in that coefficients obtained from one lift are negatives of the coefficients obtained from the other, similar to the duality observed in part (\ref{IP_dual}) of Theorem \ref{ThmIP}. 
The third lift we need for the proof of Lemma \ref{LemSqrDenoms} is the generalized Borcherds lift (see \cite{BruinierOno}) which gives the coefficients of a non-holomorphic modular function
 in terms of coefficients on a single square class of a weight $1/2$ form.

We refer the interested reader to the given references for a more detailed treatment of theta lifts and theta kernels. Here we give only the necessary theorems needed to use the lifts for our purposes. 
In section \ref{SecPLifts} we extend the two trace lifts to $p$-adic harmonic Maass forms in terms of the trace formulas for the coefficients, but we will not attempt to define a $p$-adic version of the theta kernels. 

\subsection{Two trace-lifts}
Throughout this section we will use $\Delta$ and $D$ to denote discriminants with $\Delta D<0,$ along with indices $r,s\in \Z/(2N\Z)$ satisfying 
$$(r^2, s^2)\equiv (\Delta,D) \pmod{4N}.$$ 

The first two lifts require some background on quadratic forms. Let $\mathcal Q_{N,\Delta,r}$ denote the set of positive integral binary quadratic forms $Q=[A,B,C]=Ax^2+Bxy+Cy^2$ with discriminant $\Delta$, $N\mid A$, and $B\equiv r\pmod {2N}.$ The matrix group $GL_2^+(\Q)$ acts on 
such quadratic forms by 
\[
Q(x,y)\vert \abcd = (ad-bc)^{-1}Q(ax+by,cx+dy).
\]
Let $\omega_Q$ denote the order of the stabilizer of $Q$ in $\operatorname{PSL}_2(\Z),$ and $\tau_Q$ be the CM point which is the root of $Q(x,1)$ in the upper half plane. 
We will also need the genus character $\chi_\Delta$ defined by
\[
\chi_\Delta([A,B,C])=\begin{cases}
\leg{\Delta}{n} & \text{if } \Delta\mid B^2-4AC \text{ and } 
 Q \text{ represents } n \text{ with } (n,\Delta)=1,\\
0 & \text{otherwise.}
\end{cases}
\]
If $h\in \Z/(2N\Z)$ satisfying $h^2\equiv \Delta D\pmod{4N},$ then we define the twisted modular trace $\mathbf t_{N}(\Delta,D,h)\, : \, H_0(N;\C)\to \C$ by 
\begin{equation}\label{DefCMTrace}
\mathbf t_{N}(\Delta,D,h) (F):=\sum_{Q\in\mathcal Q_{N,\Delta D,h}/ \Gamma_0(N)} \frac{\chi_{\Delta}(Q)}{\omega_Q}F(\tau_Q)
\end{equation}
which is symmetric in $\Delta$ and $D$ if these discriminants are square free. 
For the purposes of Proposition \ref{PropTransValsInt}, we modify the notation as follows.
\begin{equation}\label{eqnIntroTrace}
\widetilde{\mathbf{t}}_N(\Delta, D, F):=\frac{1}{\sqrt D} \left(\mathbf t_{N}(\Delta,D,F)
-\operatorname{sgn}(\Delta)\mathbf t_{N}(\Delta,D,-F)\right).
\end{equation}

We have the following composite theorem.

 \begin{theorem}
 \label{ThmTraceCoeffs}
Assume the notation above. Let $N$ be a square free positive integer, and let $F\in H_0(N;\C)$ have principal part at each cusp 
 given by 
\[
F|_0W_\delta=\sum_{n\leq 0}a_\delta(n)q^n+O(1)
\]
 for each divisor $\delta\mid N.$

 If $\Delta$ is fundamental, there exists a form 
\[
\vartheta^{1/2}_{\Delta,r}(F) \in H_{1/2,\tilde \rho}(N;\C),
\]
with
  $(|D|,s)$-th coefficient given by 
\[ \mathbf a_F(\Delta,r;D,s ):=\frac{1}{\sqrt D} \big(\mathbf t_{N}(\Delta,D,rs)(F)
-\operatorname{sgn}(\Delta)\mathbf t_{N}(\Delta,D,-rs)(F)\big).\]

If $D$ is fundamental, there exists a form 
\[
\vartheta^{3/2}_{D,s}(F) \in H_{3/2,\overline {\tilde \rho}}(N;\C),
\]
with $(|\Delta|,r)$-th coefficient given by
\[
\mathbf b_F(D,s;\Delta,r ):=  \frac{-1}{\sqrt D}\big(\mathbf t_{N}(D,\Delta,rs)(F)
+\operatorname{sgn}(D)\mathbf t_{N}(D,\Delta,-rs)(F)\big).
\]
Here if $\Delta >0$ then $\tilde \rho=\rho$ 
; otherwise  $\tilde \rho=\overline \rho$. 

The principal part of $\vartheta^{1/2}_{\Delta,r}(F)$ is given by
\begin{equation}\label{EqnPP1}
\vartheta^{1/2}_{\Delta,r}(F)(q)=\sum_{\delta \mid N}\sum_{n<0} a_\delta(n)\sum_{d\mid n}\leg{\Delta}{d}q^{-\frac{n^2}{4Nd^2}|\Delta|} \left(\mathfrak{e}_{\lambda_\delta\frac nd r} -\operatorname{sgn}(\Delta) \mathfrak{e}_{-\lambda_\delta\frac nd r}\right) +O(1),
\end{equation}
and the principal part of $\vartheta^{3/2}_{D,s}(F)$ is given by
\begin{equation}\label{EqnPP2}
\vartheta^{3/2}_{D,s}(F)=\sum_{\delta \mid N}\sum_{n<0} a_\delta(n)\sum_{d\mid n}\frac{n}{d} \leg{D}{d}q^{-\frac{n^2}{4Nd^2}|D|} \left(\mathfrak{e}_{\lambda_\delta\frac nd s} +\operatorname{sgn}(D) \mathfrak{e}_{-\lambda_\delta\frac nd s}\right) +O(1).
\end{equation}
Here $\mathfrak{e}_j$ is the $j$-th standard basis element with $j\in \Z/(2N\Z)$, the symbol $\leg\cdot\cdot$ is the Kronecker symbol, and $\lambda_\delta$ is defined as in (\ref{halfWtAL}).
\end{theorem}
This corrects a sign error in Theorem 4.5 and a misplaced $\sqrt{N}$ in Theorem 4.6 of \cite{ECMock} 
(These corrections can be seen in \cite{AlfesMillson} which considers a projection of this lift in the $\Delta<0$ case).
In the notation of \cite{AlfesKudlaMillson} and \cite{ECMock} respectively, the lifts $\vartheta^{3/2}$ and  $\vartheta^{1/2}$ of the theorem are given by  
\[
\vartheta^{3/2}_{D,s}(F)(\tau) = \frac{-1}{\sqrt{D}}\Lambda^e_{D,s}(\tau,F)
\]
 and 
\[
\vartheta^{1/2}_{\Delta,r}(F)(\tau) = \frac{2}{\sqrt{\Delta}}\mathcal I_{\Delta,r}(\tau,F).
\]

\begin{proof}
The theorem is almost directly a composite of Theorems 1.1, 4.3, and 4.6 of \cite{AlfesKudlaMillson} and Theorems 4.1,4.5, and 4.6 of \cite{ECMock}. The only new contribution is that we allow non-constant principal part at cusps other than infinity.

If $\mathcal S_{\Delta,r}=\mathcal Q_{N,\Delta,r}/ \Gamma_0(N)$ with quadratic forms counted with multiplicity $\frac{1}{\omega_q},$ then a short calculation (see \cite[Sec I.1 eq. (6)]{GKZ}) shows that
$$
\mathcal S_{\Delta,r}| W_\delta= \mathcal S_{\Delta, \lambda_\delta r}.
$$
This implies
\begin{equation}\label{FrickeCom}
\vartheta^{1/2}_{\Delta,r}(F|_0W_\delta)=\vartheta^{1/2}_{\Delta, \lambda_\delta r}(F),
\end{equation}
and similarly for $\vartheta^{3/2}_{D,s}.$ Since we may write 
\[
F=\sum_{\delta\mid N} F_\delta|_0W_\delta
\]
with $F_\delta$ having non-constant principal part only at the cusp infinity. The more general statement for the principal parts follows.
\end{proof}

The Hecke operators defined in the previous section also commute with these lifts. As with the Atkin--Lehner involutions, this can be seen by careful consideration of how the matrix expansions for the integer weight Hecke operators affect the cosets of quadratic forms. If $\mathcal S_{\Delta,r}$ is as above, then for a prime $p$ not dividing $N$ the Hecke operator $T_p$ acts on $\mathcal S_{\Delta, r}$ by 
$$
\mathcal S_{\Delta,r}| T_p= \mathcal S_{ p^2\Delta, pr}+\leg{\Delta}{p}\mathcal S_{ \Delta, r}+p\mathcal S_{ \frac{\Delta}{p^2}, rp^{-1}}.$$
Here the final term is taken to be $0$ if $p^2$ does not divide $\Delta$. This formula can be found for instance in \cite[Sec. I.1]{GKZ}, \cite[Proof of Theorem 5(2)]{ZagierTraces} or \cite[pgs. 290--292]{ZagierEisRiem}. Using the multiplication relations already established for the Hecke operators, we can extend these formulas for operators of non-prime index. We find that
\[ 
\vartheta^{1/2}_{\Delta,r}(\mathbf f|_0\hatT_n)=\vartheta^{1/2}_{\Delta,r}(\mathbf f)|_{1/2,\tilde \rho}\mathbf T_n
\]
and 
\[ 
\vartheta^{3/2}_{D,s}(\mathbf f|_0\hatT_n)=\vartheta^{3/2}_{D,s}(\mathbf f)|_{3/2,\tilde \rho}\mathbf T_n.
\]

The non-holomorphic parts of the lifts can be determined using the following, using linearity to extend to 
cases when $\xi_0 F$ is not a multiple of a newform. 

\begin{theorem}[{\cite[Theorem 4.3, Proposition 4.4]{ECMock}}]\label{ThmHalfShadows}
Assume the notation above, and suppose $F\in H^g_0(N;\C)$ for some newform $g\in S_{2}(N;\C)$.
 If $F$ has a zero constant term at all cusps, then the following hold.
\begin{enumerate}
\item 
The lift  $\vartheta^{3/2}_{D,s} (F)$ is weakly holomorphic and 
\[\xi_{1/2} \vartheta^{1/2}_{\Delta,r} (F)\in S_{3/2,\overline {\tilde \rho}}(N;\C)
\]
corresponds to some multiple of $\xi_0 F$ under the Shimura correspondence. \label{NonHolo1}
\item
The lift $\vartheta^{1/2}_{\Delta,r}(F)$ is weakly holomorphic if and only if 
\[L(\xi_0 F,\Delta,1)=0.\]
Here $L(\xi_0 F,\Delta,s)$ is the twisted $L$ function defined as in (\ref{TwistedL}).
   In particular, $\vartheta^{1/2}_{\Delta,r}( F)$ is weakly holomorphic if $F$ is.\label{ThmHalfS_L} 
\end{enumerate}
If we allow the constant terms of $F$ at cusps to be non-zero, then $\xi_{1/2} \vartheta_{\Delta,r}^{1/2} (F)$ and $\xi_{3/2} \vartheta^{3/2}_{D,s} (F)$ may differ from what is stated above by a finite linear combination of unary theta functions.
\end{theorem}

One result of the Shimura correspondence between $\xi_{1/2} \vartheta^{1/2}_{\Delta,r} (F)$ and $\xi_0 F$ is that if $\operatorname{sgn}(\Delta)=\epsilon_g,$ where 
$\epsilon_g=\pm 1$ is the eigenvalue of $g$ under the Fricke involution, 
\[g|_2W_N=:\epsilon_g g,\] 
then the lift $\vartheta^{1/2}_{\Delta,r}( F)$ is weakly holomorphic. In this case, it is known that ${L(\xi F,\Delta,1)=0}.$ Equivalently, we can observe that the formulas above show that $\xi_{1/2} \vartheta^{1/2}_{\Delta,r} (F)$ has eigenvalue $-\epsilon_g$ under the Fricke involution despite corresponding to $g$ under the Shimura correspondence. 

If $\vartheta^{1/2}_{\Delta,r} (F)$ is weakly holomorphic, then the algebraicity of the coefficients is tied to the algebraicity of the principal parts. On the other hand, if $L(\xi_0 F,\Delta ,1)\neq 0$, some coefficients may still be algebraic despite the form not being weakly holomorphic. 
\begin{theorem}[{\cite[Theorems 7.6, 7.8]{BruinierOno}}]\label{ThmVanishingL}
Assume the notation above, and suppose $F\in H^g_{0}(N;K,\sigma)$ for some newform $g$, so that $F$ has zero constant terms at all cusps. 
Then the following are equivalent:
\begin{enumerate}
\item The coefficient $\mathbf a_F(\Delta,r;D,s )$ is algebraic.
\item The coefficient $\mathbf a_F(\Delta,r;D,s )$ is in $K$.
\item The product 
$$L(\xi_0 F,\Delta ,1)L'(\xi_0 F, D,1)=0.$$
\end{enumerate}
 \end{theorem}

 \subsubsection{Lifts for non-fundamental discriminants}
Notice that given $F$ as above, when $D$ and $\Delta$ are both fundamental, then Theorem \ref{ThmTraceCoeffs} shows there is duality between the coefficients of $\vartheta^{1/2}_{\Delta,r}(F)$ and those of $\vartheta^{1/2}_{D,s} (F)$, similar to that seen in Theorem \ref{ThmIP}(\ref{IP_dual}). That is, in the notation of the theorem

\begin{equation*}
\mathbf a_F(\Delta,r;D,s )=-\mathbf b_F(D,s;\Delta,r ).
\end{equation*}

Notice that there are gaps in this duality statement since it requires $\Delta$ and $D$ to both be fundamental. These gaps can be filled. Using the formulas above for the action of the Hecke operators, we find that if $\ell$ is a prime not dividing $N$ and
\[
F_{\ell^m}:=F|_0\left(\hatT_{\ell^m}-\leg{D}{\ell}\hatT_{\ell^{m-1}}\right),
\]
then
\[
\mathbf b_F(D,s;\Delta \ell^{2m},r\ell^{m} ) =\mathbf b_{F_{\ell^m}}(D,s;\Delta,r )=-\mathbf a_{F_{\ell^m}}(\Delta,r;D,s ).
\]
Therefore, if we define 
\begin{equation}\label{EqnHeckeNonFund}
\vartheta^{1/2}_{\Delta\ell^{2m},r\ell^{m}} (F):=\vartheta^{1/2}_{\Delta,r}(F_{\ell_m})=\vartheta^{1/2}_{\Delta,r}|_{1/2,\tilde \rho} \left(\mathbf T_{\ell^m}-\leg{D}{p}\mathbf T_{\ell^{m-1}}\right),
\end{equation}
and define $\mathbf a_F(\Delta\ell^{2m},r\ell^{m};D,s )$ to be its $(|D|,s)$-th coefficient, we have that 
\begin{equation*}
\mathbf a_F(\Delta \ell^{2m},r\ell^m;D,s )=-\mathbf b_F(D,s;\Delta\ell^{2m},r\ell^{m} ).
\end{equation*} 

 If $\ell\mid N$, then notice that 
 \[
\mathbf t_N(\Delta, D \ell^2,h\ell) (F)=\mathbf t_{N/\ell}(\Delta, D,h) (F|_0 \operatorname {Tr_\ell}),
\]
where 
$\operatorname {Tr_\ell}$ is the trace operation defined in \ref{EqnTrace1} sending level $N$ forms to level $N/\ell$ forms. Therefore in this case it is natural to define 
\begin{equation}\label{EqnHeckeNonFund2}
\vartheta^{1/2}_{\Delta\ell^{2m+2},r\ell^{m+1}}(F):=\vartheta^{1/2}_{\Delta,r}F_{\ell_m}
\end{equation}
where 
\[
 F_{\ell_m}:=(F|_0\operatorname {Tr_\ell})|_0\left(\hatT_{\ell^m}-\leg{D}{p}\hatT_{\ell^{m-1}}\right).
\]
Here, $\hatT_n$ represents the level $N/\ell$ operator. 

 We can define $\vartheta^{1/2}_{\Delta n^2,rn} (F)$ similarly for all $n$ by requiring that (\ref{EqnHeckeNonFund}) and (\ref{EqnHeckeNonFund2}) hold as identities for all discriminants $\Delta$, not divisible by $\ell^2$. We define the lifts $\vartheta^{3/2}_{\Delta,r}F$ similarly. Notice that the different normalization in the Hecke operators corresponds to the changing denominators in the formulas in Theorem \ref{ThmTraceCoeffs}.

Using the symmetry between the action of the Hecke operators of weights $1/2$ and $3/2$, we see that 
\begin{equation}\label{EqnDualityHalf}
\mathbf a_F(\Delta,r;D,s )=-\mathbf b_F(D,s;\Delta,r )
\end{equation}
for all pairs of discriminants $\Delta D<0.$

With this notation, we have the following theorem on the algebraicity of the coefficients $\mathbf a_F(\Delta,r;D,s )$ and $\mathbf b_F(D,s;\Delta,r )$, generalizing Theorem \ref{ThmVanishingL}. \begin{theorem}\label{ThmALgebraicCross}
Assume the notation above, and suppose $F\in H^{g}_0(N;K,\sigma)$ for some newform $g\in S_2(N,\C)$, so that the constant terms of $F$ at all cusps are $0$.

If $\mathbf a_F(\Delta,r;D,s )=-\mathbf b_F(D,s;\Delta,r)$ is algebraic, then 
either $\vartheta^{1/2}_{\Delta,r}(F)$ or $\vartheta^{3/2}_{D,s}(F)$ is weakly holomorphic with every coefficient in $K$.

\end{theorem}
 \begin{proof}
If $D$ is fundamental, then this follows from applying \cite[Theorems 7.6 and 7.8]{BruinierOno} to $\vartheta^{1/2}_{\Delta,r} (F)$ for all pairs $(\Delta,r).$ On the other hand, if $D$ is not fundamental then since there exists a rational basis for $M^{!}_{3/2,\overline {\tilde \rho}}(N;K),$ there is some cusp form $\mathbf h\in S_{3/2,\overline {\tilde \rho}}(N;\C)$ such that $\vartheta^{3/2}_{D,s} (F)-\mathbf h$ has coefficients in $K$. Since $F|_0(\hatT_\ell-a_g(\ell))$ is weakly holomorphic, we have that 
\[
\vartheta^{3/2}_{D,s}\left( F|_0(\hatT_\ell-a_g(\ell))\right) =\vartheta^{3/2}_{D,s} (F)|_{3/2,\overline{\tilde \rho}} (\mathbf T_\ell-a_g(\ell))  
\]
has coefficients in $K$, and so we can take $\mathbf h$ so that it corresponds to a multiple of $g$ under the Shimura correspondence. In particular, this means that if $\mathbf h$ has any transcendental coefficients, then $\mathbf b_F(D,s;\Delta,r )$ can only be algebraic if the $(|\Delta|,r)$-th coefficient of $\mathbf h$ vanishes. But then the $(|\Delta|,r)$-th coefficient of $\vartheta^{3/2}_{D',s'} (F)$ must be algebraic for all pairs $(D',s')$. 

In this case, to see that $\vartheta^{1/2}_{\Delta,r} (F)$ is weakly holomorphic, write $\Delta=\Delta'n^2$ with $\Delta'$ fundamental, and take $D'$ to be fundamental with $L'(g,D',1)\neq 0$. Then 
\[
\vartheta^{1/2}_{\Delta,r} (F)=\vartheta^{1/2}_{\Delta',r'} (F^*)
\]
for some $F^*$, as in (\ref{EqnHeckeNonFund2}). By Theorem \ref{ThmHalfShadows}, $\vartheta^{1/2}_{\Delta,r} (F)$ must be weakly holomorphic.
 \end{proof}

\subsubsection{Borcherds lift}

We now consider the generalized Borcherds lift, $\Phi_{\Delta,r}(\mathbf f)$ as described in \cite[Sections 5-7]{BruinierOno}. 
\begin{theorem}[{\cite[Theorem 5.3]{BruinierOno}}]\label{ThmBp1}
Let $D$ be a fundamental discriminant with $s^2\equiv D\pmod{4N}.$ Suppose $\mathbf f \in H_{1/2,\tilde \rho}(N;\C)$ has principal parts at all components defined over $K$, with constant term $a^+_{\mathbf f}(0,0)=0$. Moreover, if $\Delta=1$, let $\mathbf f$ be orthogonal to the cusp forms $S_{1/2,\tilde \rho}(N;\C).$ 
Then for a specific rational number $R_{f,D}$ ($=0$ if $D\neq 1$),
the function
\[
\Phi_{D,s}(\mathbf f)(\tau):=R_{f,D,s} -4\sum_{n\geq 1}\sum_{b \pmod{\Delta}} \leg{D}{b}a^+_{\mathbf f}(|D| n^2,sn) \log|1-q^n\zeta_{D}^b|
\]
is modular for $\Gamma_0(N)$ and is harmonic on the upper half plane apart from logarithmic singularities at certain CM points determined by the principal part of $\mathbf f.$ 
\end{theorem}
The rational number $R_{f,D,s}$ here comes from 
 the Weyl vector. Notice we have included some simplifying assumptions from the original theorem. Remark 13 ii) of \cite{BruinierOno} indicates that $\Phi_{\Delta,r}(\mathbf f)$ has a similar expansion at other cusps. In fact unsurprisingly it turns out that for $\delta\mid N,$
\begin{equation}\label{PhiAL}
\Phi_{\Delta,r}(\mathbf f)|_0W_\delta=\Phi_{\Delta,r}(\mathbf f|_{1/2}\mathbf W_\delta).
\end{equation}
This follows by considering the action of the Atkin--Lehner involution $W_\delta$ on the divisor of $\Phi_{D,s}(\mathbf f)$ (see \cite[Proposition 5.2]{BruinierOno}). The calculation is nearly identical to that needed for (\ref{FrickeCom}). This is also worked out in this context in \cite[(7.3)]{BruinierOno}.

Similarly, for $(n,N)=1$ we have that
\[
\Phi_{D,s}(\mathbf f)|_0\hatT_n=\Phi_{D,s}(\mathbf f|_{1/2}\mathbf T_n).
\]
This is shown as follows. The $q$-expansion of $\Phi_{D,s}(\mathbf f)$ at infinity is obtained using the Taylor expansion 
\[\log (1-x)= \sum_{m\geq 1}\frac{-x^m}{m}.
 \]
For $n=\ell$ a prime, the identity is easily shown using formulas for the action of the operators $\hatT_\ell$ and $\mathbf T_\ell$ respectively. Since $\hatT_n$ and $\mathbf T_n$ obey the same multiplicativity relations, we see commutativity holds for general $n$.

We will also need a closely related form
, the Borcherds product
\[
\Psi_{D,s}(\mathbf f)(\tau):= \prod_{n\geq 1}\prod_{b \pmod{D}} (1-q^n\zeta_{D}^b)^{\leg{D}{b}a^+_{\mathbf f}(|D| n^2,rn) }
\]
which satisfies 
\[
\Phi_{D,s}(\mathbf f)=-4\log|\Psi_{D,s}(\mathbf f)|.
\]
Here, as usual, by $a^b$ we mean $\operatorname{e}^{b \log a},$ taking the principal branch of the logarithm. 
\begin{theorem}[{\cite[Theorems 6.1, 6.2]{BruinierOno}}]\label{ThmBorcherdsProd}
Assume the notation above and the hypotheses of Theorem \ref{ThmBp1}. Then the function
$\Psi_{D,s}(\mathbf f)(\tau)$ converges for all $\tau$ with $y$ sufficiently large, and has a meromorphic continuation to all of $\H.$
Moreover, it satisfies 
\[
\Psi_{D,s}(\mathbf f)|_0\gamma=\nu(\gamma)\Psi_{D,s}(\mathbf f)
\]
for every $\gamma\in \Gamma_0(N)$, with $\nu$ a character for $\Gamma_0(N)$ with $|\nu(\gamma)|=1.$ 

If the coefficients $a^+_{\mathbf f}(m,r)$ of $\mathbf f$ are in $\Z$ for all $h\pmod{2N}$ and $m\leq 0$, then the order of $\nu$ is finite if and only if the coefficients $a^+_{\mathbf f}(|D| n^2,sn)$ are rational for all $n$.
\end{theorem}

\subsection{Proofs of Proposition \ref{PropTransValsInt} and Lemma \ref{LemSqrDenoms}}
With this background on the trace lifts and Borcherds lifts we can now 
prove Proposition \ref{PropTransValsInt} and Lemma \ref{LemSqrDenoms}.
 
\begin{proof}[Proof of Proposition \ref{PropTransValsInt}]
If $F$ has no constant term at any cusp, then (\ref{EqnVanL1}) follows immediately from Theorem \ref{ThmVanishingL}. Similarly in this case if we fix $r,s$ so that 
$rs\equiv h\pmod {2N}$, then
the 
proof of Theorem \ref{ThmALgebraicCross} shows that there is some cusp form $\mathbf g$ corresponding to a multiple of $g$ under the Shimura correspondence so that 
$\vartheta^{3/2}_{D,s} (F)-\mathbf g$ has coefficients in $K$. 
The proposition follows.

If $F$ has constant terms at cusps, then we may subtract off a weakly holomorphic form with matching constant terms, and with coefficients in $K$. The twisted trace of CM values of this weakly holomorphic modular form must itself be in $K$, so the proposition holds in this case as well. 
\end{proof}


\begin{proof}[Proof of Lemma \ref{LemSqrDenoms}]
If the constant term $a_{\mathbf f}^+(0,0)$ of $\mathbf f$ is not $0$, we may subtract a theta series to make it so, without affecting the boundedness of the denominators on any square class. 
 For each $\sigma\in \Gal(K/\Q)$, let $\mathbf f^\sigma\in H_{1/2,\tilde \rho}(N;\C)$ have principal part defined by acting on each coefficient of the principal part of $\mathbf f$ by $\sigma$. This form is uniquely defined up to the possible addition of a cuspidal theta function. This will only concern us in the case $D=1$ since the theta series must have non-zero coefficients supported only on powers of $q^{\frac{1}{4N}}$ with square exponents. In that case, we require that each $\mathbf f^\sigma$ is also orthogonal to cusp forms. Then we have that 
$$a^+_{\mathbf f^\sigma}(1,1)=\left( a^+_{\mathbf f}(1,1)\right)^\sigma.$$ 

We will need the trace function
\[
\mathbf F_0 =\sum_\sigma \mathbf f^\sigma.
\]
By \cite[Theorem 5.5]{BruinierOno}, the coefficients $a^+_{\mathbf F}(|D|n^2,sn)$ are rational for all $n\in \Z.$

For each $\sigma\in \Gal(K/\Q),$ let $\Psi^\sigma:=\Psi_{s,D}(\mathbf f^\sigma),$ and let 
\[
 \Psi:=\Psi_{s,D}(\mathbf F_0)=\prod_{\sigma}\Psi^\sigma.
\]
Then by Theorem \ref{ThmBorcherdsProd}, there is some power $M$ so that $\widehat \Psi^M$ is a meromorphic modular function for $\Gamma_0(N),$ and so $\widehat \Psi^M$ has bounded denominators. We can calculate the  expansion of $\widehat \Psi^M|_0\gamma$ for any $\gamma\in \SL_2(\Z)$ using the factorization in (\ref{prodH}) and (\ref{PhiAL}). This allows us to calculate the function 
\[
\widehat \Psi:= \prod_{\gamma\in \Gamma_0(N)\backslash \SL_2(\Z)}\Psi^M|_0\gamma,
\]
 which has level $1$, and is therefore a rational function in $j$. In fact this function can be worked out explicitly following section 8 of \cite{BruinierOno} by considering its divisor. We find that $\widehat \Psi$ is a quotient of twisted Hilbert class functions, as defined by Zagier in \cite[Equation (2.2)]{ZagierTraces} as follows.  If $D$ and $\Delta$ are discriminants with $D\Delta<0$, then the twisted Hilbert class functions are defined by
 \[
\mathcal H_{\Delta,D}(\tau)= \prod_{Q\in\mathcal Q_{N,\Delta D}/ \SL_2(\Z)} (j(\tau)-j(\tau_Q))^{\frac{\chi_{\Delta}(Q)}{\omega_Q}}.
 \]
 These 
  functions have coefficients which are algebraic integers in $\Q(\sqrt{D})$. Moreover their inverses are their respective images under the non-trivial Galois element in $\Gal(\Q(\sqrt{D})/\Q)$. We refer the interested reader to section 8 of \cite{BruinierOno} for more details. For our purposes, it suffices to know that the coefficients of $\widehat\Psi$ are all algebraic integers. 
 
  Notice that the $q$-expansion of $ \Psi^M|_0 \gamma(q)$ for every $\gamma$ begins $1+O(q^{1/N}),$ and so given a prime $\mathfrak P$ of $\Q(\zeta_N),$ we cannot have that $\widehat \Psi^M|_0 \gamma(q)\equiv 0\pmod{\mathfrak P}.$ Therefore for each such $\gamma$, there is some integral element $\pi_\gamma$ of $\Q(\zeta_N)$ so that $\pi_\gamma\widehat \Psi^M|_0 \gamma$ has algebraic integer coefficients, but $\pi_\gamma\widehat \Psi^M|_0 \gamma(q) \not \equiv 0\pmod {\mathfrak P}.$ Then
 \[
\prod_{\gamma\in \Gamma_0(N)\backslash \SL_2(\Z)}\pi_\gamma \Psi^M|_0\gamma(q)= \widehat \Psi(q) \prod_\gamma \pi_\gamma \neq 0 \pmod {\mathfrak P}.
 \]
But since $\widehat \Psi(q)$ already has algebraic integral coefficients, we have that $\prod_\gamma \pi_\gamma$ is not divisible by $\mathfrak P$. Since $\mathfrak P$ was arbitrary,  $\Psi_0^M|_0 \gamma(q)$ has no denominators for any $\gamma$.
 
Expanding the binomial and simplifying the resulting 
 Gauss sums, we find
 \[
 \prod_{b \pmod{D}} (1-q^n\zeta_{D}^b)^{M\leg{D}{b}a^+_{\mathbf F_0}(|D|n^2,sn) }= 1- \operatorname{\sgn}(D)\sqrt{D}M a^+_{\mathbf F_0}(|D|n^2,sn)q^n +O(q^{2n}).
 \]
An induction argument on the coefficients of $\Psi^M$ shows that $\sqrt{D}M a^+_{\mathbf F_0}(|D|n^2,sn)$ is integral. Since $D$ is fundamental, the coefficients $M a^+_{\mathbf F_0}(|D|n^2,sn)$ are integers. 

This suffices to prove the theorem when $K=\Q$. For general $K$ we must continue a bit further. Suppose $K=\Q(\alpha)$ for some algebraic integer $\alpha$, and let $d$ be the degree of $K$.
 For $1\leq m<d$, define
 \[
\mathbf F_m =\sum_\sigma  (\alpha^m)^\sigma\mathbf f^\sigma.
\]
Then the coefficients in the principal part of $\mathbf F_m$ are all in $\Z$, and following the argument above, we see that there is some bound $M_m\in \Z$ so that  for all $n\in \Z$ the coefficients $M_ma^+_{\mathbf F_m}(|D|n^2,sn)$ are integers.
 
 The relation connecting the series $\mathbf f^\sigma$ and $\mathbf F_m$ is a linear transformation. Let $\mathbf M$ be the $d\times d$ matrix with row indexed by integers $m\in [0,d-1]$ and columns indexed by elements $\sigma$ of the Galois group $\Gal(K/\Q),$ such that the $(m,\sigma)$-th entry of $\mathbf M$ is given by $\sigma(\alpha^m).$ Then 
 \[
 \mathbf M \cdot (\mathbf f^\sigma)_{\sigma}=(\mathbf F_m)_m.
 \]
 The matrix $\mathbf M$ is invertible. Its left kernel is fixed by any Galois action, but any non-trivial rational relation between rows would easily translate into a polynomial over $\Q$ of degree less then $d$ satisfied by $\alpha$.
 
 Since the $\mathbf F_m$ all have rational coefficients with bounded denominators  on the given square class, the series $\mathbf f^\sigma$ can only acquire additional denominators coming from $\mathbf M^{-1}.$ The lemma follows. 
\end{proof}

\section{ The $\frakp$-adic trace lifts} \label{SecPLifts}
We will extend the trace lifts $\vartheta^{1/2}_{\Delta,r}$ and $\vartheta^{3/2}_{D, s}$ to $p$-adic harmonic Maass forms as follows. Since the lifts are linear, we will, as usual, consider the case that ${F^\sigma\in H^g_0(N;K,\sigma)}$ for some newform $g$ with $q$-expansion $\sum_{n\geq 1}b_nq^n.$ 
For primes $\frakp$ of $K$ not dividing $N$, we can adapt the appropriate lifts to the $\frakp$-adic harmonic Maass forms. However we consider the images of these lifts only in terms of $p$-adic $q$-series. To the author's knowledge, the literature does not contain a sufficiently robust theory of half-integer weight $p$-adic modular forms on which to build a theory of half-integer weight $p$-adic harmonic Maass forms. We leave this as a point for future exploration. 

CM points are necessarily in the integral locus. Recall that in this region $F^\frakp$ is defined by the $\frakp$-adic limit of the action of certain Hecke operators. Set
\begin{equation}\label{APrime}
A'_n=\begin{cases} A_n & v_p(\beta_g)<1/2\\
\frac12(A_n+\overline A_n) & v_p(\beta_g)=1/2,
\end{cases}
\end{equation}
where $A_n$ is defined in (\ref{prodH}). Then for every $E$ in the integral locus of $\mathcal E(N;K)$, we have that as a $\frakp$-adic limit,
\[
F^\frakp(E)=\lim_{n\to \infty} F^\sigma|_0A'_n(E).
\]

Since the Hecke operators commute with the lifts, we define 
 \begin{equation}\label{DefThetap}
 \vartheta^\kappa_{*,*}(F^\frakp)(q):=\lim_{n\to\infty}\vartheta^\kappa_{*,*}(F^\sigma|_0A'_n)(q)=\lim_{n\to\infty}\vartheta^\kappa_{*,*}(F^\sigma)|_{\kappa,\tilde \rho}\mathbf{A}'_n (q)
 \end{equation}
 whenever the limit converges $\frakp$-adically and coefficient-wise. Here $\vartheta^\kappa_{*,*}=\vartheta^{1/2}_{\Delta,r}$ or $\vartheta^{3/2}_{D, s}$, and $\mathbf A'_n$ is defined similar to $A'_n$, replacing the weight $0$ operators with the corresponding half-integer weight operators. The lift $\vartheta^\kappa_{*,*}$ extends to all of $H_0(N;K_\frakp)$ by linearity.
 
 It turns out that these lifts always converge and their coefficients are given by traces over a CM elliptic curve, as in the Archimedean case. Given a positive definite integer binary quadratic form $Q$, let $E_Q$ be the associated CM elliptic curve, and let $K^{\Delta,r}$  be the minimal Galois extension of $K$ with $E_Q\in \mathcal E(N;K^{\Delta,r})$ for each $Q\in\mathcal Q_{N,\Delta,r}.$ 
The function $F^\frakp$ extends to a function over $K_\frakp^{\Delta,r}$, which is unique, up to the action of the Galois group $\Gal(K^{\Delta,r}_\frakp/K_\frakp)$. If we modify equation (\ref{DefCMTrace}) to be
\begin{equation}\label{DefCMTrace2}
\mathbf t_{N}(\Delta,D,h) ~ F=\sum_{Q\in\mathcal Q_{N,\Delta D,h}/ \Gamma_0(N)} \frac{\chi_{\Delta}(Q)}{\omega_Q}F(E_Q),
\end{equation}
then this similarly modifies the definition of $\widetilde {\mathbf t_{N}}(\Delta,D,h) ( F^\frakp)$ in (\ref{eqnIntroTrace}). 
The modified $\widetilde {\mathbf t_{N}}(\Delta,D,h) ( F^\frakp)$ is well-defined and yields values in $K_\frakp$. Here we have used the fact that the set 
\[
\{E_Q \ : Q\in \mathcal Q_{N,\Delta D,h}/ \Gamma_0(N)\}
\] 
is a union of complete orbits of elliptic curves under $\Gal(\Q^{\Delta,h}/\Q)$, and the definition of the twisted trace $ \widetilde{\mathbf t}_{N}(\Delta,D,h)$ in (\ref{eqnIntroTrace})  ensures it is invariant under $\Gal(K^{\Delta,h}_\frakp/K_\frakp)$. 
 
 With this modified definition of the modular trace, we have the following theorem.
 \begin{theorem}\label{ThmThetaConv}
Let $F^\frakp\in H_0(N;K_\frakp),$ and let the lifts $ \vartheta^\kappa_{*,*} (F^\frakp)(q)$ be defined as above. Then the lift  $ \vartheta^\kappa_{*,*} (F^\frakp)(q)$ converges $\frakp$-adically and coefficient-wise. Moreover the coefficients of $ \vartheta^\kappa_{*,*} (F^\frakp)(q)$ agree with the formulas given Theorem \ref{ThmTraceCoeffs}. 
  \end{theorem}
 
 \begin{remark}
 As noted earlier, Theorem \ref{ThmCorrespondence} (\ref{ThmCorr_AlphaTrace}) will follow as a corollary and gives a (weak) analog of Theorem \ref{ThmVanishingL} in this setting. Moreover, combined with part Theorem \ref{ThmIntegrality} (\ref{ThmInt_val}), this gives explicit bounds on powers of primes allowed to appear in denominators of algebraic coefficients of half-integer weight harmonic Maass forms. 
 \end{remark}
 
 The proof of this Theorem relies on the following proposition which shows that, as in the integral weight case with Lemma \ref{VanishingHecke}, the difference between $p$-th power Hecke operators on certain forms converges $\frakp$-adically and coefficient-wise to $0$. This is nontrivial in the weight $3/2$ case, and requires Lemma \ref{LemSqrDenoms}. 
 \begin{proposition}\label{PropDiffHeckeHalf}
Suppose $\mathbf f\in M^!_{\kappa,\tilde \rho}(N;K)$ with $\kappa=1/2$ or $3/2$. 
\begin{enumerate}
\item If $\kappa=1/2,$ then as $n$ increases, the sequence of $q$-series $\left(f|_{k}(\mathbf T_{p^n}-\mathbf T_{p^{n+2}})(q)\right)_{n\in \N}$ converges coefficient-wise $p$-adically to $0$. \label{PDHH_1/2}
\item If $g\in S_2(N;\C)$ is a newform and $\mathbf f=\vartheta^{3/2}_{D,s} F$ for some $F\in H^g_0(N;\C)$, then 
 the sequence of $q$-series $\left(\mathbf f|_{3/2,{\tilde\rho}}(\mathbf T_{p^n}-\mathbf T_{p^{n+2}})(q)\right)_{n\in \N}$ converges coefficient-wise $p$-adically to $0$.  
  \label{PDHH_3/2}
\end{enumerate}
\end{proposition} 

\begin{proof}[Proof of Proposition \ref{PropDiffHeckeHalf}]
The proof of part (\ref{PDHH_1/2}) involves a short exercise using formula (\ref{EqnTHalf2}). Write $D=D'p^{2m}$ with $p^2$ not dividing $D'.$ Then only terms $\mathbf U_a \mathbf S_b\mathbf V_c$ of the Hecke operators with $c\mid p^m$ contribute to the $(D,s)$-th coefficient. Moreover, the terms $a\mathbf U_a \mathbf S_b\mathbf V_c$ of the $ \mathbf T_{p^{n}}$ cancel with terms $a\mathbf U_a \mathbf S_{bp^2}\mathbf V_c$ from the $ \mathbf T_{p^{n+2}}$ whenever $b>1$. Therefore the contributing terms $a\mathbf U_a \mathbf S_b\mathbf V_c$ of the Hecke operators have $b$ and $c$ bounded. As $n$ increases, so must the power of $p$ dividing $a$. Since $\mathbf f$ has bounded denominators, the contribution from these terms $p$-adically goes to $0$. 

Part (\ref{PDHH_3/2}) relies essentially on the calculations behind (\ref{EqnDualityHalf}) and Lemma \ref{LemSqrDenoms}. If $D$ is not fundamental, then recall that we defined the lift $\vartheta^{3/2}_{D,s}(F)$ in terms of Hecke operators which commute with the lift. In particular, We have 
\[\vartheta^{3/2}_{D,s}(F)=\vartheta^{3/2}_{D',s'}(F'),
\]
for some $F'$, and with $D'$ fundamental. Note that $\xi_0 F'=\alpha\xi_0 F$ for some (possibly zero)  $\alpha \in K$ since $\xi_0 F$ is an eigenform for the Hecke operators. In particular, $F'\in H^g_0(N;\C)$. Therefore we will assume without loss of generality that $D$ is fundamental.
Let 
\[F_n:=F|_0 (\hatT_{p^n}-\hatT_{p^{n+2}}),\]
then we have that 
\[
\mathbf a_{F_n}(\Delta,r;D,s )=-\mathbf b_{F_n}(D,s;\Delta,r ).
\]
If $\mathbf h:=\vartheta^{1/2}_{\Delta,r} (F)|_{1/2,\overline {\tilde \rho}}$, then the left-hand side is a coefficient of  $\mathbf h|_{1/2,\overline{\tilde \rho}} \left(\mathbf T_{p^n}-\mathbf T_{p^{n+2}}\right).$ If $\mathbf h$ is weakly holomorphic, then the statement follows immediately from part (\ref{PDHH_1/2}).

If $\mathbf h$ is not weakly holomorphic, we can still adapt this argument. 
From Lemma \ref{LemSqrDenoms} we know that the coefficients $\mathbf{a}_{F_n}(\Delta,r;Dn^2,sn )\in K$ for all $n$, and have bounded denominators (here we have used that $D$ is fundamental). Following the argument above for part (\ref{PDHH_1/2}), but restricting our attention only to this square class, we see that the specified coefficient $\mathbf {a}_{F_n}(\Delta,r;D,s )$ converges $\frakp$-adically to zero as $n$ increases.  
\end{proof}

We now complete the proof of Theorem \ref{ThmThetaConv}.
 \begin{proof}[Proof of Theorem \ref{ThmThetaConv}]
 Equation (\ref{DefThetap}) shows that the trace formulas given in Theorem \ref{ThmTraceCoeffs} hold, so long as the coefficients converge. Given Proposition \ref{PropDiffHeckeHalf}, convergence follows essentially as in Section \ref{SecConstruction}, \emph{mutatis mutandis}. 
 Proposition \ref{PropDiffHeckeHalf} and (\ref{prodH}) show that the principal part converges to the same shape. In fact we see that if our starting form is weakly holomorphic, the limit has the same $q$-expansion at all components. We will come back to this point in the proof of Theorem \ref{ThmCorrespondence} (\ref{ThmCorr_AlphaTrace}).
 \end{proof}~

\begin{proof}[Proof of Theorem \ref{ThmCorrespondence} (\ref{ThmCorr_AlphaTrace})] ~
 Following the argument in the proof of Proposition \ref{PropTransValsInt}, we reduce to the case that $F^\nu$ has vanishing constant terms at all cusps. Ler $h=rs$ be any factorization of $h$ so that $r^2\equiv \Delta \pmod(4N)$ and $s^2\equiv D \pmod{4N}.$
 
Theorems \ref{ThmVanishingL} and \ref{ThmALgebraicCross} show that if ${L(\xi_0F^\sigma,\Delta,1)L'(\xi_0F^\sigma,D,1)=0},$ then either $\vartheta^{1/2}_{\Delta,r}(F^\sigma)$ or  $\vartheta^{3/2}_{D,s}(F^\sigma)$ is weakly holomorphic with coefficients in $K$. Then (\ref{DefThetap}) and Proposition \ref{PropDiffHeckeHalf} show that the coefficients of $\vartheta^\kappa_{*,*}(F^\frakp)$ are the same as those of $\vartheta^\kappa_{*,*}(F^\sigma)$.

More generally, we can consider $\vartheta^{3/2}_{D,h}(F^\frakp)$ and $\vartheta^{3/2}_{D,h}(F^\sigma).$ The proof is then little more than that of Theorem \ref{ThmCorrespondence} (\ref{ThmCorr_AlphaCoeff}), using the argument in the proof of Theorem \ref{ThmALgebraicCross} to translate to this setting.  If $\mathbf g\in S_{3/2,\overline {\tilde \rho}}(N;K)$ maps to $g$ under the Shimura correspondence, 
then there must be constants $\boldsymbol\alpha^\sigma_{D}\in \C$ and $\boldsymbol\alpha^\frakp_{D}\in K_\frakp$ so that

\[
\left(\vartheta^{3/2}_{D,h}(F^\sigma)(q)  - \boldsymbol\alpha^\sigma_{D}\mathbf g(q)\right) =
\left(\vartheta^{3/2}_{D,h}(F^\frakp)(q)  - \boldsymbol\alpha^\frakp_{D}\mathbf g(q)\right)\in K((q)).
\]
The claim follows by considering individual coefficients.

\end{proof}

\bibliography{refs}{}

\begin{thebibliography}{10}

\bibitem{AlfesKudlaMillson}
Claudia Alfes.
\newblock Formulas for the coefficients of half-integral weight harmonic
  {M}aa\ss\ forms.
\newblock {\em Math. Z.}, 277(3-4):769--795, 2014.

\bibitem{ECMock}
Claudia Alfes, Michael Griffin, Ken Ono, and Larry Rolen.
\newblock Weierstrass mock modular forms and elliptic curves.
\newblock {\em Res. Number Theory}, 1:1:24, 2015.

\bibitem{AlfesMillson}
Claudia Alfes-Neumann and Markus Schwagenscheidt.
\newblock On a theta lift related to the {S}hintani lift, {P}reprint.
\newblock {\em {A}rxiv: https://arxiv.org/abs/1605.07054}.

\bibitem{Bo1}
Richard~E. Borcherds.
\newblock Automorphic forms with singularities on {G}rassmannians.
\newblock {\em Invent. Math.}, 132(3):491--562, 1998.

\bibitem{MockAsPAdic}
Kathrin Bringmann, Pavel Guerzhoy, and Ben Kane.
\newblock Mock modular forms as {$p$}-adic modular forms.
\newblock {\em Trans. Amer. Math. Soc.}, 364(5):2393--2410, 2012.

\bibitem{PAdicCouplingHalf}
Kathrin Bringmann, Pavel Guerzhoy, and Ben Kane.
\newblock Half-integral weight p-adic coupling of weakly holomorphic and
  holomorphic modular forms.
\newblock {\em Res. Number Theory}, 1:1:26, 2015.

\bibitem{BringmannOnoMaassPoincare}
Kathrin Bringmann and Ken Ono.
\newblock Arithmetic properties of coefficients of half-integral weight
  {M}aass-{P}oincar\'e series.
\newblock {\em Math. Ann.}, 337(3):591--612, 2007.

\bibitem{BruinierOno}
Jan Bruinier and Ken Ono.
\newblock Heegner divisors, {$L$}-functions and harmonic weak {M}aass forms.
\newblock {\em Ann. of Math. (2)}, 172(3):2135--2181, 2010.

\bibitem{BruinierFunke}
Jan~H. Bruinier and Jens Funke.
\newblock On two geometric theta lifts.
\newblock {\em Duke Math. J.}, 125(1):45--90, 2004.

\bibitem{BruinierFunkeTraces}
Jan~H. Bruinier and Jens Funke.
\newblock Traces of {CM} values of modular functions.
\newblock {\em J. Reine Angew. Math.}, 594:1--33, 2006.

\bibitem{ThetaLiftings}
Jan~H. Bruinier, Jens Funke, and {\"O}zlem Imamo{\=g}lu.
\newblock Regularized theta liftings and periods of modular functions.
\newblock {\em J. Reine Angew. Math.}, 703:43--93, 2015.

\bibitem{BruinierOno2}
Jan~H. Bruinier and Ken Ono.
\newblock Algebraic formulas for the coefficients of half-integral weight
  harmonic weak {M}aass forms.
\newblock {\em Adv. Math.}, 246:198--219, 2013.

\bibitem{BOR}
Jan~H. Bruinier, Ken Ono, and Robert~C. Rhoades.
\newblock Differential operators for harmonic weak {M}aass forms and the
  vanishing of {H}ecke eigenvalues.
\newblock {\em Math. Ann.}, 342(3):673--693, 2008.

\bibitem{BruinierStein}
Jan~H. Bruinier and Oliver Stein.
\newblock The {W}eil representation and {H}ecke operators for vector valued
  modular forms.
\newblock {\em Mathematische Zeitschrift}, 264:249--270, 2010.

\bibitem{Candelori}
Luca Candelori.
\newblock Towards a p-adic theory of harmonic weak {M}aass forms.
\newblock {\em M.Sc. Thesis, McGill University}, 2010.

\bibitem{CandeloriGeometric}
Luca Candelori.
\newblock Harmonic weak {M}aass forms of integral weight: a geometric approach.
\newblock {\em Math. Ann.}, 360(1-2):489--517, 2014.

\bibitem{CandeloriCastella}
Luca Candelori and Francesc Castella.
\newblock A geometric perspective on p-adic properties of mock modular forms.
\newblock {\em Res. Math. Sci.}, 4:4:5, 2017.

\bibitem{DIT1}
W.~Duke, \"O. Imamo{\=g}lu, and \'A. T\'oth.
\newblock Cycle integrals of the {$j$}-function and mock modular forms.
\newblock {\em Ann. of Math. (2)}, 173(2):947--981, 2011.

\bibitem{DIT2}
W.~Duke, \"O. Imamo{\=g}lu, and \'A. T\'oth.
\newblock Real quadratic analogs of traces of singular moduli.
\newblock {\em Int. Math. Res. Not. IMRN}, (13):3082--3094, 2011.

\bibitem{DukeJenkins}
W.~Duke and Paul Jenkins.
\newblock Integral traces of singular values of weak {M}aass forms.
\newblock {\em Algebra Number Theory}, 2(5):573--593, 2008.

\bibitem{Dwork}
B.~Dwork.
\newblock The {$U_{p}$} operator of {A}tkin on modular functions of level
  {$2$}\ with growth conditions.
\newblock pages 57--67. Lecture Notes in Math., Vol. 350, 1973.

\bibitem{GKZ}
B.~Gross, W.~Kohnen, and D.~Zagier.
\newblock Heegner points and derivatives of {$L$}-series. {II}.
\newblock {\em Math. Ann.}, 278(1-4):497--562, 1987.

\bibitem{ZagiersAdele}
Pavel Guerzhoy.
\newblock On {Z}agier's adele.
\newblock {\em Res. Math. Sci.}, 1:Art. 7, 19, 2014.

\bibitem{PAdicCoupling}
Pavel Guerzhoy, Zachary~A. Kent, and Ken Ono.
\newblock {$p$}-adic coupling of mock modular forms and shadows.
\newblock {\em Proc. Natl. Acad. Sci. USA}, 107(14):6169--6174, 2010.

\bibitem{Hida}
Haruzo Hida.
\newblock Iwasawa modules attached to congruences of cusp forms.
\newblock {\em Ann. Sci. \'Ecole Norm. Sup. (4)}, 19(2):231--273, 1986.

\bibitem{Hovel}
M.~H\"ovel.
\newblock Automorphe formen mit {S}ingularit{\''a}ten auf dem hyperbolischen
  {R}aum.
\newblock 2012.

\bibitem{KaneCong}
Ben Kane and Matthias Waldherr.
\newblock Explicit congruences for mock modular forms.
\newblock {\em J. Number Theory}, 166:1--18, 2016.

\bibitem{Katz}
Nicholas~M. Katz.
\newblock {$p$}-adic properties of modular schemes and modular forms.
\newblock pages 69--190. Lecture Notes in Mathematics, Vol. 350, 1973.

\bibitem{MillerPixton}
Alison Miller and Aaron Pixton.
\newblock Arithmetic traces of non-holomorphic modular invariants.
\newblock {\em Int. J. Number Theory}, 6(1):69--87, 2010.

\bibitem{Niebur}
D.~Niebur.
\newblock A class of nonanalytic automorphic functions.
\newblock {\em Nagoya Math. J.}, 52:133--145, 1973.

\bibitem{Serre}
Jean-Pierre Serre.
\newblock Formes modulaires et fonctions z\^eta {$p$}-adiques.
\newblock pages 191--268. Lecture Notes in Math., Vol. 350, 1973.

\bibitem{Silverman1}
Joseph~H. Silverman.
\newblock {\em The arithmetic of elliptic curves}, volume 106 of {\em Graduate
  Texts in Mathematics}.
\newblock Springer, Dordrecht, second edition, 2009.

\bibitem{SZ}
Nils-Peter Skoruppa and Don Zagier.
\newblock Jacobi forms and a certain space of modular forms.
\newblock {\em Invent. Math.}, 94(1):113--146, 1988.

\bibitem{Walds}
J.-L. Waldspurger.
\newblock Sur les coefficients de {F}ourier des formes modulaires de poids
  demi-entier.
\newblock {\em J. Math. Pures Appl. (9)}, 60(4):375--484, 1981.

\bibitem{ZagierEisRiem}
Don Zagier.
\newblock Eisenstein series and the {R}iemann zeta function.
\newblock In {\em Automorphic Forms, Representation Theory and Arithmetic},
  pages 275--301. Springer-Verlag Berlin–Heidelberg–New York, 1981.

\bibitem{ZagierTraces}
Don Zagier.
\newblock Traces of singular moduli.
\newblock In {\em Motives, polylogarithms and {H}odge theory, {P}art {I}
  ({I}rvine, {CA}, 1998)}, volume~3 of {\em Int. Press Lect. Ser.}, pages
  211--244. Int. Press, Somerville, MA, 2002.

\end{thebibliography}
\bibliographystyle{plain}

\end{document}